\documentclass[fleqn,10pt]{wlscirep}
\usepackage[utf8]{inputenc}
\usepackage[T1]{fontenc}
\usepackage{bm}
\usepackage[mathlines]{lineno} 
\usepackage{subcaption}
\usepackage{tabularx,booktabs,array}
\usepackage{amsthm}

\newtheorem{definition}{Definition}
\newtheorem{theorem}{Theorem}
\newtheorem{Lemma}{Lemma}
\newtheorem{proposition}{Proposition}

\title{Frequency-induced Tipping: Resonance-Driven Tipping in a Threshold Predator-Prey System}

\author[1,*]{Shan Gao}
\author[2]{Silas Poloni}
\author[1]{Hao Wang}
\author[3,4]{Mark Lewis}
\affil[1]{Department of Mathematical and Statistical Sciences, University of Alberta, Edmonton, Alberta, Canada}
\affil[2]{Department of Science and Environment, Roskilde University, Roskilde, Denmark}
\affil[3]{Department of Mathematics and Statistics, University of Victoria, Victoria, British Columbia, Canada}
\affil[4]{Department of Biology, University of Victoria, Victoria, British Columbia, Canada}

\affil[*]{e-mail: sgao9@ualberta.ca}

\begin{abstract}
Tipping points refer to abrupt, substantial, and often irreversible transitions in dynamic systems that can be triggered by minor perturbations. Existing tipping mechanisms are usually classified by the nature of the trigger: bifurcation-induced tipping occurs when a parameter crosses a critical threshold, noise-induced tipping is caused by random perturbations, rate-induced tipping occurs when a system fails to track a moving attractor, and phase-induced tipping is triggered by timed perturbations relative to an oscillatory state. In this paper, we identify a new mechanism for tipping: frequency-induced tipping, in which periodic forcing drives a transition through resonant amplification. This occurs by destabilization of an equilibrium via subthreshold perturbations of the appropriate frequency. Filippov analysis of a piecewise-smooth model identifies admissible crossings that lead either to predator loss or to extinction of predator and prey, and numerical simulations show that tipping persists after smoothing of the thresholds. These results extend the current tipping framework to include critical transitions under recurrent environmental forcing.
\end{abstract}
\begin{document}

\flushbottom
\maketitle

\thispagestyle{empty}

\section{Introduction}
Tipping, or a critical transition, refers to sudden, significant, and often difficult-to-reverse changes in the state of a dynamical system that can be triggered by small changes or perturbations~\cite{folke2004regime,scheffer2001catastrophic}. Such transitions are widely documented in ecological systems, where nonlinear feedbacks and alternative attractors can cause ecosystems to respond gradually over some ranges of environmental change but suddenly shift once resilience is sufficiently eroded~\cite{scheffer2001catastrophic}. Classic ecological tipping examples include eutrophication-driven transitions in shallow lakes from clear, macrophyte-dominated states to turbid, phytoplankton-dominated states; vegetation collapse and desertification in drylands; and coral reef phase shifts from coral-dominated to macroalgae-dominated communities~\cite{scheffer2001catastrophic,mumby2007thresholds,hughes2010rising,van2002spatial}. Understanding the dynamical mechanisms that generate such tipping phenomena is central to assessing resilience, predicting ecosystem collapse, and designing effective management strategies.

Although critical transitions, by themselves, may appear phenomenologically similar, the mechanisms that generate them can differ substantially~\cite{ashwin2012tipping, screen2010central}. In the context of mathematical theory, four major mechanistic scenarios have been proposed to explain critical transitions in dynamical systems. Previous studies have often focused on the sudden shifts that occur when gradual quantitative changes in system conditions (i.e. parameters) alter the underlying stability structure. This mechanism is referred to as bifurcation-induced tipping or B-tipping~\cite{scheffer2001catastrophic, scheffer2009early, kuehn2011mathematical}. A classic example is hysteresis, where the system shifts from a base attractor to an alternative stable state as an external condition gradually crosses a bifurcation threshold. Once this transition has occurred, simply reversing the external condition may not restore the original state; instead, a larger decrease or increase in the parameter is required to push the system back to its base attractor; see Figure~\ref{fig:Tipping}A.

Crossing a bifurcation threshold, however, is not necessary for tipping. For instance, stochastic disturbances can perturb the system state or its parameters, driving a trajectory across a basin boundary into the domain of attraction of an alternative regime. If such disturbances cause the system to deviate sufficiently toward a different basin of attraction, the system can shift to another attractor, as illustrated in Figure~\ref{fig:Tipping}B. This phenomenon is known as noise-induced tipping (N-tipping)~\cite{horsthemke1984noise, ditlevsen2010tipping}. In contrast to B-tipping, which requires sufficient qualitative changes, the rate of change in the parameter can also induce rapid shifts in system states, referred to as rate-induced tipping (R-tipping)~\cite{wieczorek2011excitability, wieczorek2023rate}. R-tipping occurs when the system fails to track a moving attractor under sufficiently rapid parameter change, often without crossing a bifurcation threshold (Figure~\ref{fig:Tipping}C). For systems exhibiting stable limit-cycles, phase-induced tipping (P-tipping) has been proposed as a means to destabilise the dynamics~\cite{alkhayuon2021phase, alkhayuon2023stochastic}. P-tipping occurs when a parameter shift displaces a system from a limit-cycle into an alternative basin of attraction during the vulnerable phase of the cycle, even though the shift does not exceed the bifurcation threshold (Figure~\ref{fig:Tipping}D). 

Tipping mechanisms described above distinguish several ways in which the temporal structure and stochasticity of external forcing can affect system stability. Yet they do not fully address a common feature of environmental variability: repeated forcing at a characteristic frequency. Many environmental drivers fluctuate periodically or quasi-periodically across a wide range of temporal scales~\cite{fine2015ecological}. At short timescales, diel cycles in light and temperature entrain circadian regulation and shape daily physiological and behavioral rhythms in many organisms~\cite{paranjpe2005evolution, cohen2015circadian}; intermediate-scale cycles include seasonal variation in temperature, precipitation, resource availability, and life-history timing~\cite{white2020seasonality}; and longer-timescale variability includes climate modes such as the El Ni\~{n}o-Southern Oscillation (ENSO), the North Atlantic Oscillation (NAO), and the Pacific Decadal Oscillation (PDO), whose periods can span years and influence ecological processes across broad spatial scales~\cite{stenseth2003studying}. These examples suggest that environmental change is often structured not only by its magnitude or direction, but also by its characteristic frequency.

Periodic forcing can alter both the occurrence and interpretation of tipping in non-autonomous systems~\cite{medeiros2017trapping,bathiany2018abrupt}. In such systems, tipping may involve transitions between time-dependent attracting responses rather than between static equilibria, and tipping occurrence can depend on the amplitude, period, frequency, and phase of the external forcing~\cite{alkhayuon2023stochastic}. Among the properties of periodicity, resonance provides a natural mechanism for this frequency dependence. For instance, periodic environmental variation can be amplified by nonlinear ecological responses, producing resonance-like population outbreaks~\cite{blarer1999resonance}, while stochastic resonance has been shown to arise near a critical transition in a lake-eutrophication model~\cite{ma2022relaxation}. Stochastic resonance has also been linked to increased phase-tipping risk in cyclic predator–prey systems under climate variability~\cite{alkhayuon2023stochastic}. These results imply that environmental fluctuations with the same mean and amplitude may have qualitatively different effects at different frequencies. This motivates the following hypothesis:
\begin{quote}
\itshape
Subthreshold forcing can induce tipping when applied at a resonant frequency.
\end{quote}
In this paper, we identify and characterize frequency-induced tipping (F-tipping): tipping occurs when external periodic forcing resonates with a system’s intrinsic dynamics and amplifies endogenous variation enough to drive the trajectory out of the basin of attraction of its base attractor, although that regime remains stable for every frozen forcing value within the forcing range. We compare F-tipping with other tipping mechanisms in Table~\ref{tab:tipping_mechanisms}. 

This perspective has direct implications for ecosystem management, especially in systems whose dynamics are governed by intrinsic timescales such as growth, mortality, dispersal, recovery, resource renewal, and species interactions. Management strategies often focus on reducing the magnitude of environmental stressors, slowing their rate of change, or preventing extreme perturbations~\cite{folke2004regime,scheffer2009early,ashwin2012tipping,wieczorek2023rate}. F-tipping suggests an additional control dimension: the recurrence frequency of environmental variation. Identifying frequencies that resonate with intrinsic ecological timescales could help managers avoid destabilizing intervention schedules, anticipate periods of heightened vulnerability, and design restoration or disturbance regimes that minimize resonant amplification. Preventing or promoting transitions may require not only asking how much, how fast, or when the environment changes, but also how often it changes.

\begin{figure}[ht]
    \centering
    \includegraphics[width=\linewidth]{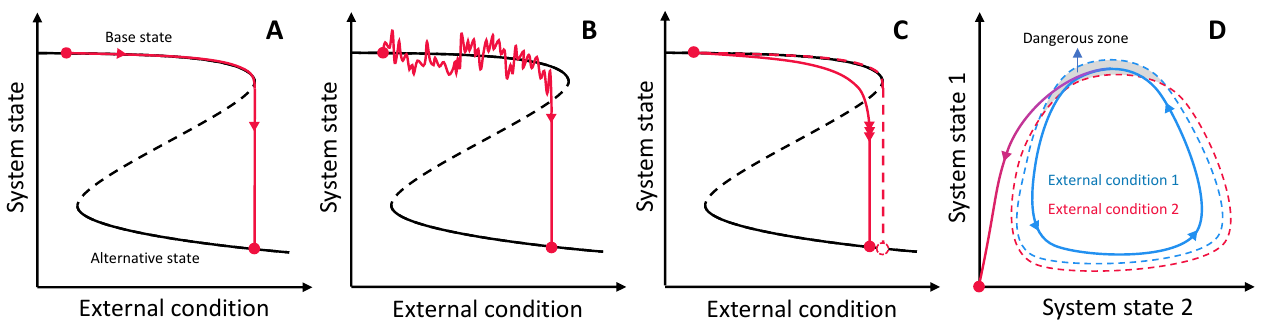}
    \caption{Schematic illustration of four tipping mechanisms. (A) Bifurcation-induced tipping occurs when external changes push the system across a critical threshold. (B) Noise-induced tipping arises when stochastic fluctuations act directly on the system's state, pushing it past a critical threshold and triggering a shift between stable states, even when external conditions remain constant. (C) Rate-induced tipping can be triggered when rapidly changing external forcing deflects the system into an alternative stable state. (D) Phase-induced tipping occurs in oscillatory dynamical systems when small perturbations in external conditions affect certain phases (or timing). Dashed curves indicate the basin boundaries. If the external condition changes while the system is in the "dangerous zone", the original oscillation can no longer be maintained.}
    \label{fig:Tipping}
\end{figure}

\begin{table}[htbp]
\centering
\small
\renewcommand{\arraystretch}{1.2}
\begin{tabularx}{\textwidth}{
>{\raggedright\arraybackslash}p{0.10\textwidth}
>{\raggedright\arraybackslash}p{0.35\textwidth}
>{\raggedright\arraybackslash}X}
\toprule

\multicolumn{3}{p{\dimexpr\textwidth-2\tabcolsep\relax}}{
\textbf{Proposed F-tipping:}
Periodic forcing drives a transition through resonant
amplification of the system's response.
} \\

\midrule
\textbf{Existing mechanism}
& \textbf{Primary trigger}
& \textbf{Distinction from F-tipping} \\
\midrule

B-tipping
& A slowly varying parameter crosses a bifurcation threshold at which
the original attractor loses stability or disappears.
& F-tipping can occur even when the periodic forcing remains within bounds for which the system does not tip under fixed parameter values. \\

N-tipping
& Random fluctuations push the state or parameters across a basin boundary or critical threshold. 
& F-tipping is driven by deterministic periodic forcing rather than stochastic perturbations. The transition depends on forcing frequency, not noise amplitude or rare random events. \\

R-tipping
& Non-autonomous instability in which the system loses the ability to track a changing base attractor, without crossing a frozen-system bifurcation.
& F-tipping emphasizes resonant amplification at selected
forcing frequencies, rather than tracking failure alone. \\

P-tipping
& Perturbations applied at vulnerable phases of an existing oscillation drive the system away from its original attractor. 
& F-tipping does not require a pre-existing stable limit cycle or perturbation at a dangerous phase. Its defining
feature is frequency-selective amplification, although
the external forcing phase can also affect the outcome. \\

\bottomrule
\end{tabularx}
\caption{Comparison of commonly discussed tipping mechanisms
with the proposed frequency-induced tipping (F-tipping),
emphasizing their primary triggers and distinctions.}
\label{tab:tipping_mechanisms}
\end{table}

\section{Results}
\subsection{Filippov predator-prey model}
We first introduce the autonomous predator-prey system that provides the unforced baseline for the frequency-induced tipping analysis:
\begin{equation}
\begin{pmatrix}
\dfrac{dN}{dt} \\[6pt]
\dfrac{dP}{dt}
\end{pmatrix}=
\begin{pmatrix}
rN\left(1-\dfrac{N}{K}\right)\theta(N-N_a)-rN\theta(N_a-N)-aNP\theta(P-P_a) \\[6pt]
baNP\theta(P-P_a)-mP
\end{pmatrix},
\label{predator_prey_model}
\end{equation}
with the Heaviside step function:
\begin{equation}
\theta(x)=
\begin{cases}
1, & x>0,\\
0, & x\le 0.
\end{cases}
\label{step_fun}
\end{equation}
In this model, $N(t)$ and $P(t)$ denote the prey and predator populations at time $t$. The prey dynamics are governed by their intrinsic growth rate $r$, carrying capacity $K$, and Allee effect threshold $N_a$. For predators, parameters $a$, $b$, $m$, and $P_a$ denote the predator attack rate, conversion efficiency, natural mortality rate, and the predator Allee effect threshold, respectively.

The Heaviside functions $\theta(\cdot)$ introduce abundance-dependent switches in prey growth and predation, representing idealized Allee effects in both populations. For prey, the intrinsic population dynamics follow logistic growth when $N>N_a$ and decline exponentially otherwise. This switch captures demographic disadvantages at low abundance, such as mate limitation, reduced group vigilance, and weaker collective defence~\cite{gascoigne2009dangerously, mooring2004vigilance}. $N_a$ therefore represents the threshold below which the prey population cannot sustain itself even in the absence of predation. For predators, $P_a$ represents the minimum abundance required for effective predation, as might arise from dependence on cooperative hunting. When $P\leq P_a$, predation is inactive and predators decline solely through natural mortality. When $P>P_a$, predation removes prey at rate $aNP$ and contributes to predator reproduction at rate $baNP$. These switches make system~\eqref{predator_prey_model} piecewise-smooth, with discontinuities along $N=N_a$ and $P=P_a$, and motivate its treatment within the Filippov framework~\cite{bernardo2008piecewise}.

Assuming $N_a>0$ and $P_a>0$, these thresholds partition the nonnegative phase plane $\mathbb{R}_+^2:=\{(x,y):x\geq0, y\geq0\}$ into four distinct regions $R_i$, $i\in\{1,2,3,4\}$, as shown in Figure~\ref{fig:figure2}A. Within each region, system~\eqref{predator_prey_model} is governed by a smooth vector field $f_i$. The boundaries separating adjacent regions are denoted by $\Sigma_{(i,j)}$, with $(i,j)\in\{(1,2),(2,3),(3,4),(4,1)\}$. The vector field therefore switches whenever a trajectory intersects one of these boundaries. The full subsystem formulation and formal Filippov definitions are detailed in~\nameref{sec:method}. 

Unless otherwise specified, we use the baseline parameters in~\eqref{eqn:parameters}. The unforced system~\eqref{predator_prey_model} has three biologically relevant candidate equilibria:
\begin{equation}\label{eqn:equilibria}
    E_0=(0,0), \qquad 
    E_1=(K,0), \qquad
    E_2= (N^*,P^*) = \left(\frac{m}{ab},\frac{r}{a}\left(1-\frac{m}{baK}\right)\right),
\end{equation}
corresponding to extinction, prey-only persistence, and predator-prey coexistence, respectively. Note that the coexistence equilibrium $E_2$ is always regular ($N^*>N_a$ and $P^*>P_a$, see~Definition\ref{Def:regular}) under our parameter choices. We numerically computed basins of attraction, as shown in Figure~\ref{fig:figure2}B, with representative simulations illustrated in Figure~\ref{fig:figure2}C-E. The basins of attraction show that if the trajectory crosses the prey threshold $N_a$, the dynamics can lead to extinction $E_0$; if it crosses the predator threshold $P_a$ while prey remain above its Allee threshold, the dynamics can lead to prey-only persistence $E_1$. Stability analyses are provided in the Supplementary Material~\ref{S:subsystems_dynamics}. Specifically, we compute the eigenvalues of the Jacobian matrix of subsystem $f_1$ evaluated at $E_2$ 
\begin{equation}
\lambda_{\pm} = -\frac{rm}{2abK} \pm i\omega, \qquad \text{where }\omega = \sqrt{mr\left(1-\frac{m}{baK}\right)-\frac{r^2m^2}{4a^2b^2K^2}}.
\end{equation}
Since $\Re(\lambda_{\pm}) < 0$ and $\Im(\lambda_{\pm}) \neq 0$ for appropriate parameter choices (see Supplemental Material \ref{sec: app-resc}), the coexistence equilibrium $E_2 = (N^*,P^*)$ is a stable focus. The natural period of the associated damped oscillations is $\tau_n=\frac{2\pi}{\omega}$. This intrinsic period motivates the mechanism studied below: periodic forcing of the prey growth rate can resonate with the damped predator-prey oscillations when the forcing period is close to $\tau_n$, amplifying the trajectory such that a switching boundary is reached, triggering a critical transition.

\begin{figure}
    \centering
    \includegraphics[width=\linewidth]{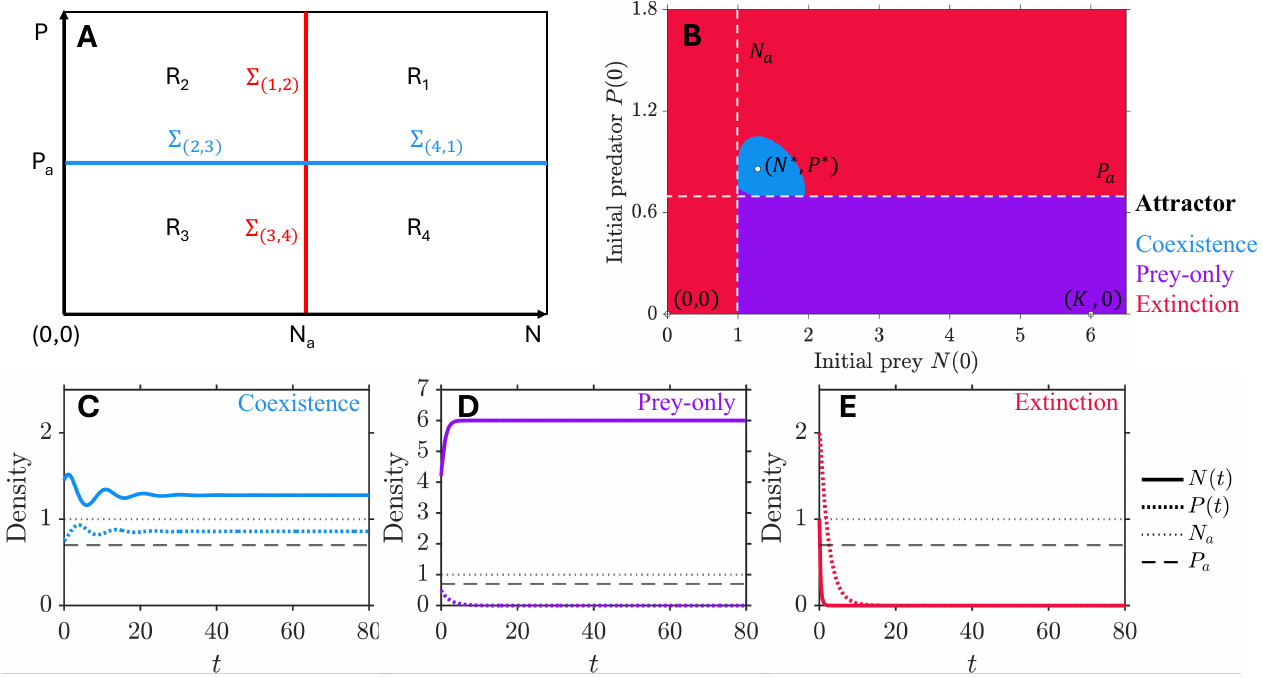}
    \caption{Baseline dynamics of the autonomous threshold predator-prey system. (A) Schematic diagram for the Allee threshold of system~\ref{predator_prey_model}. (B) Basins of attraction for the Filippov system~\ref{predator_prey_model} with Allee thresholds (in white dashed lines) $N_a=1$ and $P_a=0.7$. Initial conditions may approach extinction $E_0=(0,0)$, prey-only $E_1=(K,0)$, or coexistence $E_2=(N^*,P^*)$. Parameters used are $r=1.2$, $K=6$, $a=1.1$, $b=0.32$, and $m=0.45$. (C-E) Representative simulations for the three outcomes.}
    \label{fig:figure2}
\end{figure}

\subsection{Periodic forcing induces frequency-selective transition}
Environmental variation, including seasonal changes, can cause temporal variation in demographic rates~\cite{fretwell1972populations}, which is commonly represented using sinusoidal forcing~\cite{cushing1977periodic}. We model the periodically varying intrinsic growth rate as
\begin{equation}
    \rho(t)=r_0+A\sin\left(\frac{2\pi t}{\tau}+\phi\right)\in[r_0-A,r_0+A],
    \label{varying_rho}
\end{equation}
where $r_0$ is the mean growth rate, $A$ is the forcing amplitude, $\tau$ is the forcing period, and $\phi$ is the initial phase. Here $0\leq A<r_0$, ensuring that $\rho(t)$ stays positive. Replacing the constant growth rate $r$ in system~\eqref{predator_prey_model} with $\rho(t)$ gives
\begin{equation}
\begin{pmatrix}
\dfrac{dN}{dt} \\[6pt]
\dfrac{dP}{dt}
\end{pmatrix}
=
\begin{pmatrix}
\rho(t)N\left(1-\dfrac{N}{K}\right)\theta(N-N_a)
-\rho(t)N\theta(N_a-N)
-aNP\theta(P-P_a)
\\[8pt]
baNP\theta(P-P_a)-mP
\end{pmatrix}.
\label{predator_prey_model_varying_r}
\end{equation}
Here we choose $r_0=1.2$ and $\phi=0$, and initialise the system at the coexistence equilibrium $E_2(r_0)$. For these parameters, the frozen system admits a coexistence equilibrium for $r>r_c\approx0.978484$ (Methods~\ref{Bifurcation_diagram}), and no frozen-system bifurcation threshold was crossed since $\min_t\rho(t)=r_0-A>r_c$.

To visualize the resonant response before accounting
for threshold switching, we simulate the smooth subsystem $f_1$ in $R_1$ with $\rho(t)$. For each fixed forcing amplitude $A\in\{0,0.05,0.10,0.15\}$, we vary the forcing period over $\tau\in[0.05,20]$ and record the historical minimum and maximum populations over $t\in[0,250]$ (Figure~\ref{fig:figure3}A,B). All amplitudes satisfied the above condition, with the largest amplitude giving $\min_t\rho(t)=1.05>r_c$. Results show pronounced amplification for forcing periods $\tau$ near the natural period of the unforced coexistence equilibrium $\tau_n=\frac{2\pi}{\omega}$. Increasing $A$ raises the historical maximum and lowers the historical minimum, widening the range of forcing periods over which the prey or predator minimum fell below its corresponding threshold $N_a$ or $P_a$. In the full piecewise system, such crossings activate different dynamical regimes and can lead to population collapse.

The full-system simulations illustrate how the forcing period determines the outcome at a fixed amplitude $A=0.10$ (Figure~\ref{fig:figure3}C-G). The constant-growth controls $\rho\equiv r_0+A$ and $\rho\equiv r_0-A$ both exhibit coexistence (Figure~\ref{fig:figure3}C,D). Periodic forcing with $\tau=7.5$, which is far away from the natural period, also maintains coexistence, with both populations oscillating above their respective thresholds (Figure~\ref{fig:figure3}E). Increasing the period to $\tau=8.0$ drives the prey population below $N_a$, followed by the collapse of both populations (Figure~\ref{fig:figure3}F). By contrast, at $\tau=11.4$, the system converges to the prey-only equilibrium (Figure~\ref{fig:figure3}G). Thus, periodic variation within the same growth-rate interval that does not exceed the bifurcation threshold, produce coexistence, extinction, or prey-only equilibrium, depending on the forcing period. The amplification near $\tau_n$ and the threshold crossings in the full system support a mechanism of frequency-selective collapse, where forcing near the intrinsic oscillation timescale amplifies population fluctuations sufficiently to trigger critical transitions in the dynamics.

\begin{figure}[ht]
    \centering
    \includegraphics[width=\linewidth]{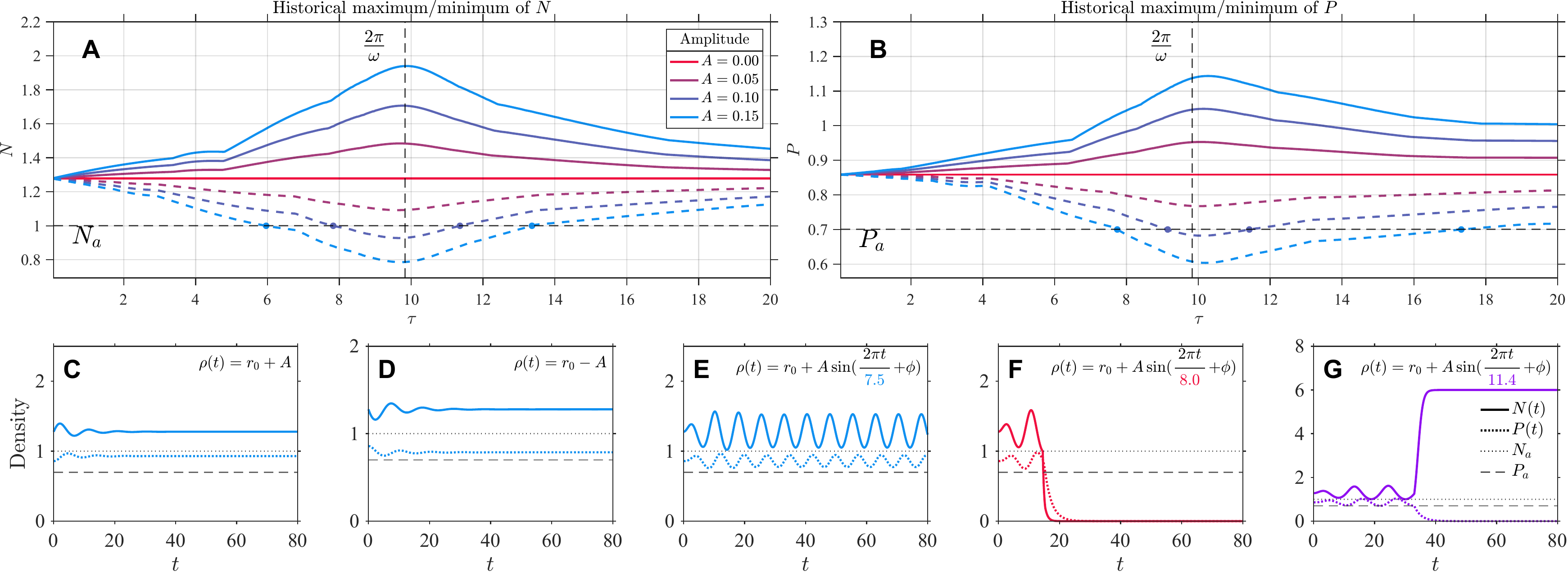}
    \caption{Resonant amplification and frequency-selective dynamics of system~\eqref{predator_prey_model_varying_r} under periodic prey-growth forcing $\rho(t)$. (A,B) Historical maximum (solid curves) and minimum (dashed curves) of prey density $N$ and predator density $P$, respectively, over $t\in[0,250]$, obtained by continuing the smooth subsystem $f_1$ without threshold switching under periodic forcing. Colours indicate the forcing amplitude $A\in\{0,0.05,0.10,0.15\}$. Horizontal dashed lines mark the thresholds $N_a$ and $P_a$, and the vertical dashed line marks the intrinsic period $\frac{2\pi}{\omega}\approx 9.83$. (C-G) Time series of the full piecewise system~\eqref{predator_prey_model_varying_r} over $t\in[0,80]$, with $A=0.10$, $r_0=1.2$, and $\phi=0$. (C,D) Constant-growth controls with $\rho=r_0+A$ and $\rho=r_0-A$, respectively. (E-G) Periodic forcing with $\tau=7.5$, $8.0$, and $11.4$, respectively. Blue indicates coexistence, red indicates extinction, and purple indicates prey-only.}
    \label{fig:figure3}
\end{figure}

\subsection{Threshold crossings determine the tipping outcome}
Periodic forcing can amplify population oscillations sufficiently to reach an Allee threshold when we only consider the subsystem $f_1$. The resulting transition, however, depends on which threshold the trajectory crosses first. Crossing the prey threshold can lead to extinction of both populations, whereas crossing the predator threshold while prey remain above their Allee threshold can lead to prey-only persistence. We now identify the crossing conditions that distinguish these outcomes.

We conduct the Filippov analysis for the autonomous system~\eqref{predator_prey_model} and show that all nonempty sliding segments of the autonomous system are escaping (see~\nameref{sec:method}), where adjacent vector fields point away from the switching boundary. The explicit ranges of the escaping sliding segments are shown in~\eqref{eq:sigma-e}. Specifically, the escaping segment on $\Sigma_{(4,1)}$ is independent of $r$, whereas the escaping segment on $\Sigma_{(1,2)}$ depends on $r$. The trajectory cannot cross these escaping segments but can cross through an admissible crossing segment when one exists. When a trajectory reaches a switching manifold $\Sigma_{(i,j)}$ through a nonsliding portion, it crosses into the adjacent smooth region according to the local vector field (see Figure~\ref{fig4}D and Figure~\ref{fig:FigureS1}B). If the reaching point lies on an escaping sliding segment, the adjacent vector fields point away from the switching manifold. Pseudo-equilibria (equilibria on the sliding segment but are not equilibria of any of the vector fields of the original system~\cite{bernardo2008piecewise}) on escaping segments are not attractors of the full Filippov dynamics (see Proposition~\ref{Pseudo}). With the Filippov analysis, we can provide a criterion of Frequency-induced tipping.
\begin{theorem}[Frequency-induced tipping criterion]
Assume that $\rho(t)>0$ for all $t$, and suppose that the initial condition $(N_0,P_0)\in R_1$. Define the values (intervals) of forcing periods $\tau$ for which the historical minimum of $N$ and $P$ fall below the Allee thresholds $N_a$ and $P_a$
\begin{equation*}
\mathcal{T}_N(A) = \left\{
\tau>0:\min_{t\in[0,\infty)}N_\tau(t)\leq N_a\right\},
\qquad
\mathcal{T}_P(A) = \left\{
\tau>0:\min_{t\in[0,\infty)}P_\tau(t)\leq P_a\right\}.
\end{equation*}
A necessary condition for frequency-induced tipping is therefore
\begin{equation*}
    \exists\tau_{\mathrm{tip}}\in \mathcal{T}_N(A)\cup\mathcal{T}_P(A).
\end{equation*}
Define the first admissible crossing time
\begin{equation}
t_c=
\inf\left\{
t>0:
\begin{array}{l}
N(t)=N_a
\ \text{and}\
P(t)>\dfrac{\rho(t)}{a}\left(1-\dfrac{N_a}{K}\right), \quad \text{or} \quad P(t)=P_a
\ \text{and}\
N_a<N(t)<\dfrac{m}{ba}
\end{array}
\right\}.
\end{equation}
If $t_c<\infty$, then frequency-induced tipping occurs. Moreover, the tipping outcome is determined by the first admissible threshold crossing:
\begin{equation}\label{criterion}
    \begin{cases}
    N(t_c)=N_a,\quad P(t_c)>\dfrac{\rho(t_c)}{a}\left(1-\dfrac{N_a}{K}\right),
    & \text{then the trajectory crosses from $R_1$ into $R_2$ and tips toward $E_0$,}\\[8pt]
    P(t_c)=P_a,\quad N_a<N(t_c)<\dfrac{m}{ba},
    & \text{then the trajectory crosses from $R_1$ into $R_4$ and tips toward $E_1$.}
\end{cases}
\end{equation}
\end{theorem}

As an example, here we choose $A=0.12$ for a more obvious difference between intervals $\mathcal{T}_N(A)$ and $\mathcal{T}_P(A)$. We numerically compute the intervals $\mathcal{T}_N(A=0.12)=[7.616540, 12.051099]$ and $\mathcal{T}_P(A=0.12)=[8.602175, 12.726177]$. This yields four possible regimes:
\begin{enumerate}[label=\textbf{Case \arabic*:},ref=Case~\arabic*,labelindent=\parindent,leftmargin=*,itemsep=0.4em]
\item \label{case1}
$\tau \notin \mathcal{T}_N(A) \cup \mathcal{T}_P(A)$: the system maintains $E_2$ and does not tip to either $E_0$ or $E_1$ (See Figure~\ref{fig4}B);
\item \label{case2}
$\tau \in \mathcal{T}_N(A)$ but $\tau \notin \mathcal{T}_P(A)$: the system can tip from $E_2$ to $E_0$ (See Figure~\ref{fig4}C);
\item \label{case3}
$\tau \in \mathcal{T}_N(A) \cap \mathcal{T}_P(A)$: the system tips from $E_2$ to either $E_0$ or $E_1$, as determined by~\eqref{criterion};
\item \label{case4}
$\tau \notin \mathcal{T}_N(A)$ but $\tau \in \mathcal{T}_P(A)$: the system tips from $E_2$ to $E_1$ (See Figure~\ref{fig4}D).
\end{enumerate}

\begin{figure}[ht]
    \centering
    \includegraphics[width=\linewidth]{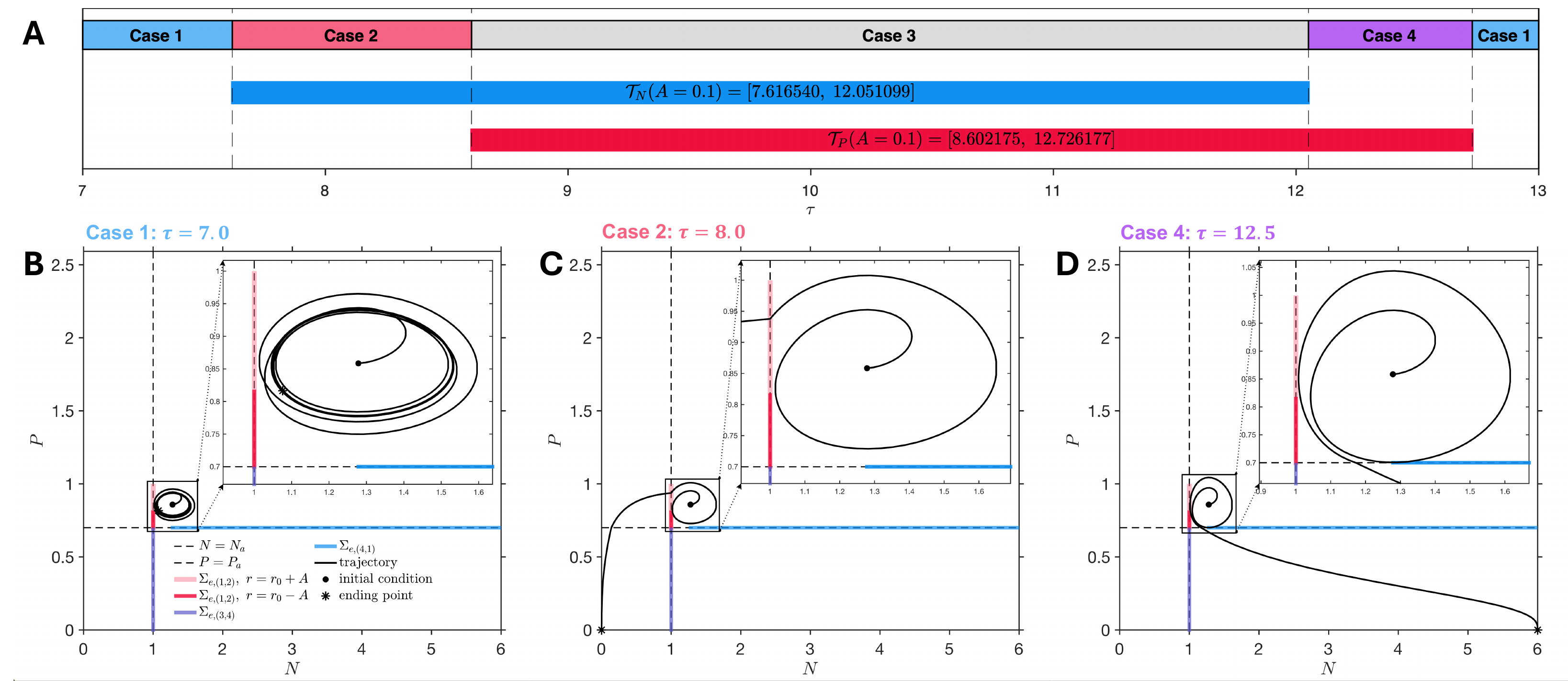}
    \caption{Threshold accessibility and tipping outcomes under periodic forcing. (A) Numerically estimated period sets $\mathcal{T}_N$ and $\mathcal{T}_P$ for $A=0.12$, partitioned into four cases according to which thresholds the unconstrained response reaches. (B-D) Representative phase portraits for $\tau=7.0$, $8.0$, and $12.5$, illustrating coexistence, total extinction, and prey-only persistence, respectively. Insets show the initial excursions near coexistence and the switching boundaries. Coloured boundary segments denote escaping sliding regions; the two red segments correspond to the frozen values $r_0\pm A$. Parameters are $r_0=1.2$, $K=6$, $a=1.1$, $b=0.32$, $m=0.45$, $N_a=1$, $P_a=0.7$, and $\phi=0$.}
    \label{fig4}
\end{figure}

For the parameter values in~\eqref{eqn:parameters}, coexistence is an admissible equilibrium when $r>r_c\approx0.978484$ (see Equation~\eqref{eq:critical_r_c}). Thus $r_0=1.2>r_c$ and take the initial condition near the coexistence equilibrium $E_2(r_0))$. For the frozen system~\eqref{predator_prey_model}, if the growth rate is held constant at $r=r_0\pm A<r_c$, the trajectory remains in the basin of attraction of $E_2$ as shown in Figure~\ref{fig:figure3}C,D.

In summary, frequency-induced tipping requires three ingredients: the unforced trajectory must remain attracted to the coexistence state, periodic forcing must amplify the response enough to reach an Allee threshold, and the threshold encounter must occur on an admissible crossing portion of the switching boundary rather than on an escaping sliding segment. The following theorem formalizes these requirements.

\subsection{Frequency, phase, and amplitude shape the tipping outcome}
Although we focus on the period $\tau$ of the varying growth rate $\rho(t)$, the other forcing parameters, amplitude $A$ and phase $\phi$, can also play a role in the tipping phenomenon. Figure~\ref{fig:figure5} illustrates the long-term outcome for different combinations of these parameters in $\rho(t)$. In each panel, the amplitude $A\in\{0.08,0.09,0.10,0.11,0.12,0.13,0.14,0.15\}$ is fixed, no bifurcation threshold is exceeded, the radial coordinate represents the forcing period $\tau\in[6,14]$, and the angular coordinate represents the phase $\phi\in[0,2\pi]$. The colours indicate whether the trajectory remains near coexistence $E_2$; tips to the prey-only equilibrium $E_1$; or tips to the extinction equilibrium $E_0$.

As the amplitude $A$ increases, the blue coexistence region shrinks, while the transition regions expand. This agrees with the expansion of the critical interval $\mathcal{T}_N(A) \cup \mathcal{T}_P(A)$: larger forcing amplitudes make tipping possible over a wider range of periods. These results also suggest that stronger periodic fluctuations in $\rho(t)$ can promote predator collapse and, in some cases, complete extinction, even though the mean forcing $r_0$ remains unchanged.

The phase $\phi$ further determines which transition occurs within the critical period ranges. One can consider changing the initial value of $\rho(t)$ with different $\phi$. When $\tau$ falls outside the transition intervals, as in \ref{case1}, the trajectory remains near coexistence for all phases. However, when $\tau$ satisfies one of \ref{case2}-\ref{case4}, changing $\phi$ shifts the timing of the forcing relative to the system state: if the trajectory crosses $\Sigma_{e,(1,2)}$ first, the system will tip to $E_0$; if the trajectory crosses $\Sigma_{e,(4,1)}$ first, the system will tip to $E_1$. Thus, while $\tau$ determines the degree of resonant amplification, $A$ controls the size of the parameter region where tipping is possible, and $\phi$ influences the direction of tipping once the system is in a critical period range.

\begin{figure}[ht]
    \centering
    \includegraphics[width=\linewidth]{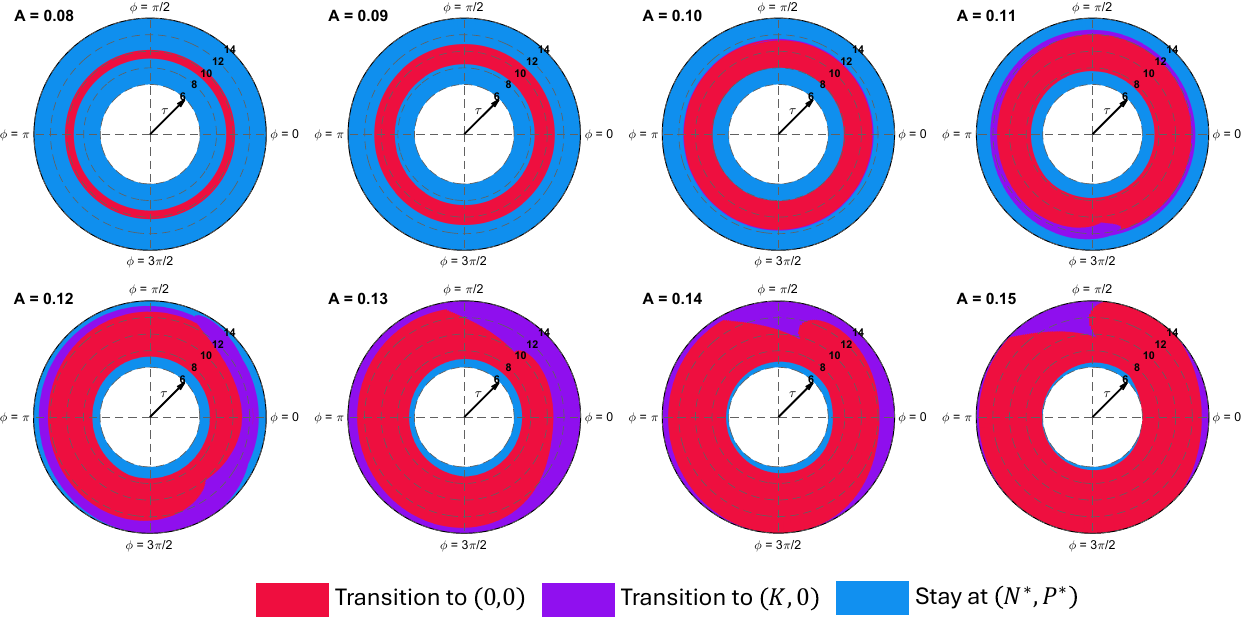}
    \caption{Critical transitions in system~\eqref{predator_prey_model_varying_r} from coexistence $(N^*,P^*)$ to either a prey-only ($(K,0)$, purple) or extinction ($(0,0)$, red), as well as cases with no transition, where coexistence persists (blue). Here, the intrinsic growth rate varies periodically: $\rho(t) = r_0 + A\sin(2 \pi t/\tau + \phi)$. Panels correspond to different values of the forcing amplitude $A$. The angular position around each circle represents the phase $\phi \in [0,2\pi]$, while the distance from the center represents the period $\tau \in [6,14]$. Parameter values are $r_0 = 1.2$, $P_a = 0.7$, $N_a = 1$, $K = 6$, $a = 1.1$, $b = 0.32$, and $m = 0.45$, yielding $\tau_n \approx9.83$.}
    \label{fig:figure5}
\end{figure}

\subsection{Frequency-selective tipping persists with smooth population responses}
We show that frequency-induced tipping can occur in the threshold predator-prey system with discontinuous switching boundaries. To examine whether this phenomenon persists after smoothing the discontinuities, we replace the Heaviside step function in~\eqref{step_fun} with the logistic function
\begin{equation}
    \gamma(x)=\frac{e^{\alpha x}}{1+e^{\alpha x}}=1-\gamma(-x),
    \label{logistic_fun}
\end{equation}
which provides a smooth approximation of the switching process~\cite{boukal2002single}. The parameter $\alpha>0$ controls the sharpness of the transition: as $\alpha\to\infty$, $\gamma(x)$ approaches the Heaviside function for $x\neq 0$.

For a fixed prey growth rate $\rho(t)\equiv r$, the smooth system shows the biologically relevant extinction, prey-only, and coexistence states (see Supplementary~\ref{continuous_approximation_analysis}). In contrast to the discontinuous system, the precise locations of the equilibria depend on $\alpha$ and are determined numerically. For the chosen parameter set~\eqref{parameter_continuous}, the stable coexistence equilibrium is a focus, with a local natural period of approximately $9.08$ (Supplementary~\ref{continuous_approximation_analysis}). Figure~\ref{Figure6} shows the simulated trajectories together with the corresponding frozen-system nullclines and equilibria. Filled circles indicate stable equilibria, whereas hollow circles indicate unstable or saddle equilibria.

In both constant environments, $\rho(t)=r_0+A$ and $\rho(t)=r_0-A$, the coexistence equilibrium remains locally asymptotically stable, and trajectories from the same initial condition converge to the respective coexistence equilibria (Figure~\ref{Figure6}A,B). However, under periodic forcing~\eqref{varying_rho}, the long-term outcome depends strongly on the forcing period $\tau$. For $\tau=9.0$, close to the local natural period of the unforced system, the trajectory leaves the coexistence region and approaches extinction of both populations (Figure~\ref{Figure6}D). For $\tau=12.3$, the predator population instead approaches extinction while the prey population approaches the carrying capacity $K$ (Figure~\ref{Figure6}E). These simulations demonstrate that frequency-induced tipping persists in the smooth approximation, despite the absence of discontinuous switching boundaries.

\begin{figure}[ht]
    \centering
    \includegraphics[width=\linewidth]{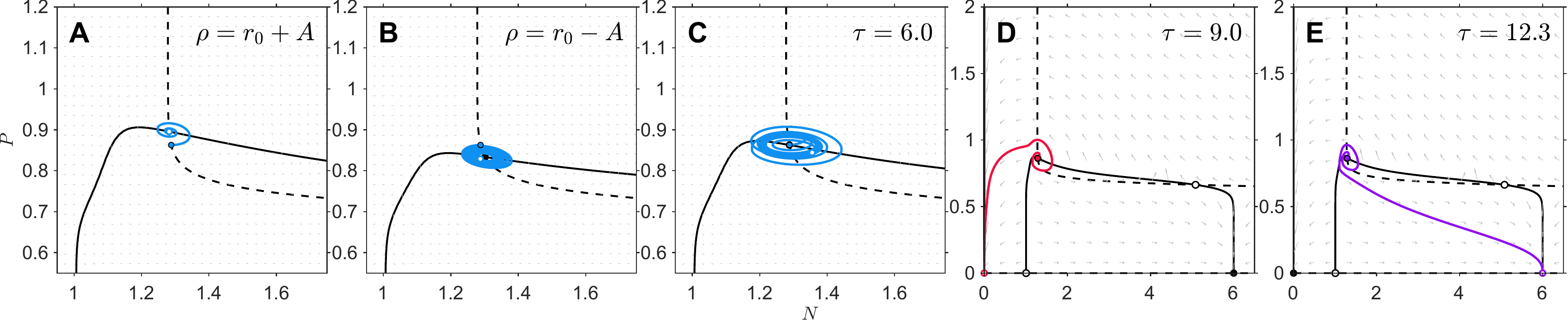}
    \caption{Phase-plane trajectories of the smoothed threshold predator--prey system under five choices of the prey growth rate $\rho(t)$. Solid black and dashed black curves denote the frozen-system nullclines $\dot N=0$ and $\dot P=0$, respectively; light-gray arrows show the corresponding frozen vector field. Filled black circles indicate stable equilibria, whereas hollow black circles indicate saddle or unstable equilibria. The coloured curves show the simulated trajectories in the $(N,P)$ plane. (A) $\rho=r_0+A$; (B) $\rho=r_0-A$; and (C-E) periodic forcing with $\tau=6.0$, $\tau=9.0$, and $\tau=12.3$, respectively. For Panels C-E, nullclines and vector fields are evaluated for the frozen reference system $\rho=r_0$. Trajectories in Panels A-C are blue, that in Panel D is red, and that in Panel E is purple. Simulations use $r_0=1.2$, $A=0.05$, $K=6.0$, $a=1.1$, $b=0.32$, $m=0.45$, $N_a=1.00$, $P_a=0.70$, $\alpha=30$, and $\phi=0$. The common initial condition is the locally asymptotically stable coexistence equilibrium of the frozen system with $\rho=r_0$.}
    \label{Figure6}
\end{figure}

\section{Discussion}
How sudden state changes occur is of great interest in ecology and dynamical studies~\cite{scheffer2001catastrophic, ashwin2012tipping}. Different tipping mechanisms can produce similar observational features in time-series data, such as abrupt transitions between alternative states (threshold behaviour~\cite{abrams2026integrating}), yet they arise from fundamentally different causes~\cite{screen2010central}. The tipping literature has extensively classified these triggers: tipping can be driven by, but not limited to, the loss of an attractor (B-tipping~\cite{scheffer2001catastrophic, scheffer2009early, kuehn2011mathematical}), failure to track changing conditions (R-tipping~\cite{wieczorek2011excitability, wieczorek2023rate}), random disturbances (N-tipping~\cite{horsthemke1984noise, ditlevsen2010tipping}), or perturbations applied at vulnerable phases (timing) of an oscillation (P-tipping~\cite{alkhayuon2021phase}). Furthermore, a system can remain far from a stable state for an extended period before eventually approaching it, a phenomenon known as long transient dynamics~\cite{hastings2018transient}. Such dynamics can produce changes that resemble critical transitions, even when no critical transition occurs. Understanding the exact mechanism behind a sudden change is essential for accurate anticipation and effective management~\cite{ritchie2016early, sardanyes2024identifying}.

In this work, we demonstrate that the stability under constant environmental conditions does not necessarily ensure persistence when those same conditions recur in a periodic fashion. Using a predator–prey system with population thresholds (a piecewise-smooth system~\cite{bernardo2008piecewise}), we identify the frequency of periodic variation, or equivalently its period, as a new critical driver of tipping, which we term Frequency-induced tipping (F-tipping). Specifically, periodic variation in prey growth amplifies damped oscillations near the coexistence equilibrium. If the forcing period aligns with the system's intrinsic dynamics, these amplified excursions cross a critical threshold, leading to predator loss or total extinction. F-tipping is therefore driven by two coupled components: frequency-selective amplification and an admissible threshold crossing that determines the subsequent outcome. 

This phenomenon is not an artifact of piecewise-smooth dynamic models. We show that the tipping phenomenon persists when discontinuous switches are replaced with smooth response functions. The piecewise-smooth formulation makes the routes out of coexistence analytically accessible, while the smooth analogue indicates that comparable transitions also occur without abrupt switching. A related example appears in the coupled human-ecological model~\cite{abbott2024and, heggerud2022coupling}, where periodic changes in the net cost of cooperation produce different cooperation outcomes at different forcing periods, although the parameter remains within the bistable range. These results are interpreted through rate dependence and nonlinear averaging; our framework explicitly identifies resonant amplification as the underlying dynamical cause.

To isolate this resonance mechanism, we intentionally employed a deterministic formulation. This is because introducing noise would conflate N-tipping and F-tipping. However, real-world ecosystems are inherently noisy, and under a stochastic framework, frequency interacts with noise to produce well-documented phenomena like stochastic resonance~\cite{mcdonnell2009stochastic, quinn2026tipping}, coherence resonance~\cite{palenzuela2001coherence, kilduff2019frequency, white2024marine}, and cohort resonance in age-structured populations~\cite{bjornstad2004trends}.  

The effects of recurrent environmental forcing have a long history in dynamical systems and ecological research. Recent work has examined how periodic forcing alters tipping near a drifting saddle-node bifurcation~\cite{zhu2015tipping}, and how the magnitude and recurrence interval of discrete disturbances on a state variable determine resilience in flow–kick systems~\cite{meyer2018quantifying}. Particularly relevant to our study, frequency-selective partial tipping has been observed using ensemble-based analysis, which quantifies changes in the destinations of trajectories across the state space~\cite{basak2024partial}. Yet how the selectivity on frequency relates to timescales of external forcing variation remains unquantified. Resonance behaviour in tipping under chaotic forcing is further identified in a low-order ice age model~\cite{quinn2026tipping}. Our work connects resonant amplification around a stable coexistence equilibrium to explicit conditions for two competing collapse routes in a threshold predator–prey system. Starting from the coexistence equilibrium, changing the forcing period can lead to either predator loss or extinction of both populations, with the outcome determined by which admissible threshold is crossed first as resonance amplifies the population response. The mathematical contribution is connecting the population response to this choice of collapse route. This distinction is mechanistically informative because the frozen coexistence prey density, $N^* = m/(ab)$, is independent of the growth rate: periodic growth variation can nevertheless drive prey below its persistence threshold through the coupled population response. These studies establish that the temporal structure of environmental variation can be as consequential as its magnitude (changes in mean values).

Our results emphasize the impact of the state variable due to the resonant frequency in the external force, and the system will therefore transition to an alternative (often unfavourable) state. The tipping outcome also depends on the phase of the periodic forcing, linking this behaviour to P-tipping. Such phenomenon cannot be fully explained by existing tipping mechanism, yet of importance for tipping studies. 

Recognising the mechanism underlying a tipping event is essential for identifying effective ecological management strategies. B-tipping motivates keeping environmental pressures below thresholds at which the desired attractor loses stability or disappears, as in limiting harvesting pressure to avoid collapse~\cite{may1977thresholds}. R-tipping emphasises the importance of the path and speed of parameter variation. For example, a rapid increase in fishing pressure can cause population collapse even when a gradual increase to the same target permits persistence~\cite{yang2025fast}. N-tipping highlights the need for safety margins against extreme events. Additionally, P-tipping shows that the timing of disturbance (at a vulnerable phase) can play a role in the Canada lynx and snowshoe hare population collapse~\cite{alkhayuon2023stochastic}. F-tipping adds a further consideration: under periodic forcing, ecological persistence can depend critically on the frequency of environmental variation relative to the system's intrinsic timescales. Thus, risk assessment does not solely rely only on the magnitude of external forcing, the distance from a bifurcation threshold, or how rapidly the forcing is applied, but can also account for how frequently the forcing is applied so to avoid the resonance effect.

Our work also raises questions about how to anticipate these transitions. Early warning signals based on critical slowing down (CSD) are most directly motivated by weakened local stability near certain bifurcations~\cite{dakos2024tipping}. However, neither R-, N-, nor P-tipping produces the CSD signals~\cite{ritchie2021overshooting, alkhayuon2021phase}. The demonstrated ability of deep learning to anticipate R-tipping also suggests scope for prediction beyond conventional CSD indicators~\cite{huang2024deep}. Future work could therefore focus on developing frequency-specific anticipation tools that look beyond classical variance metrics to detect resonant amplification before the threshold is crossed.

\section{Methods}
\label{sec:method}
\subsection{Filippov predator-prey model with Allee effect}
Using $\mathbb{R}_+^2:=\{(x,y):x\geq0, y\geq0\}$, system~\eqref{predator_prey_model} with~\eqref{step_fun} can be written away from the switching boundaries as
\begin{equation}
\begin{pmatrix}
\dfrac{dN}{dt} \\[6pt]
\dfrac{dP}{dt}
\end{pmatrix}
=
\begin{cases}
f_1(t,(N,P)), & \text{for } (N,P) \in R_1 = \{(N,P)\in\mathbb{R}^2_+:N>N_a,\ P>P_a\},\\
f_2(t,(N,P)), & \text{for } (N,P) \in R_2 = \{(N,P)\in\mathbb{R}^2_+:N<N_a,\ P>P_a\},\\
f_3(t,(N,P)), & \text{for } (N,P) \in R_3 = \{(N,P)\in\mathbb{R}^2_+:N<N_a,\ P<P_a\},\\
f_4(t,(N,P)), & \text{for } (N,P) \in R_4 = \{(N,P)\in\mathbb{R}^2_+:N>N_a,\ P<P_a\},
\end{cases}
\label{rewritten_system}
\end{equation}
where
\begin{equation*}
\begin{aligned}
f_1(t,(N,P))&=
\binom{rN\left(1-\frac{N}{K}\right)-aNP}{baNP-mP},
&
f_2(t,(N,P))&=
\binom{-rN-aNP}{baNP-mP},
\\[0.8em]
f_3(t,(N,P))&=
\binom{-rN}{-mP},
&
f_4(t,(N,P))&=
\binom{rN\left(1-\frac{N}{K}\right)}{-mP}.
\end{aligned}
\end{equation*}
Furthermore, the discontinuity boundary consists of the following four parts,
\begin{equation*}
\begin{aligned}
\Sigma_{(1,2)}&= \overline{R}_1 \cap \overline{R}_2 = \left\{ (N,P) \in \mathbb{R}_+^2 : N = N_a,\ P > P_a \right\},\\
\Sigma_{(2,3)}&= \overline{R}_2 \cap \overline{R}_3 = \left\{ (N,P) \in \mathbb{R}_+^2 : N < N_a,\ P = P_a \right\},\\
\Sigma_{(3,4)}&= \overline{R}_3 \cap \overline{R}_4 = \left\{ (N,P) \in \mathbb{R}_+^2 : N = N_a,\ P < P_a \right\},\\
\Sigma_{(4,1)}&= \overline{R}_4 \cap \overline{R}_1 = \left\{ (N,P) \in \mathbb{R}_+^2 : N > N_a,\ P = P_a \right\}.
\end{aligned}
\end{equation*}
Throughout this paper, unless otherwise specified, we use the baseline parameter set
\begin{equation}
r=1.2,\qquad
K=6,\qquad
a=1.1,\qquad
b=0.32,\qquad
m=0.45,\qquad
N_a=1,\qquad
P_a=0.7.
\label{eqn:parameters}
\end{equation} 
For the continuous system with logistic function~\eqref{logistic_fun}, we set $\alpha=30$.

The following definitions for a generic planar Filippov system~\eqref{rewritten_system} are necessary in this paper~\cite{kuznetsov2003one, di2008discontinuity}.
\begin{definition}\label{Def:regular}
A point $E^*$ in $\mathbb{R}_+^2$ is called a regular equilibrium of system~\eqref{rewritten_system} if there exists $i \in \{1,2,3,4\}$ such that $f_i(E^*)=0$ and $E^* \in R_i$.
\end{definition}

\begin{definition}\label{Def:virtual}
A point $E^*$ in $\mathbb{R}_+^2$ is called a virtual equilibrium of system~\eqref{rewritten_system} if there exists $i \in \{1,2,3,4\}$ such that $f_i(E^*)=0$ and $E^* \in R_j$ where $j \in \{1,2,3,4\}$ and $\ j \ne i$.
\end{definition}

\begin{definition}
For each $(i,j)\in\{(1,2),(2,3),(3,4),(4,1)\}$, let
$n_{(i,j)}(x)$ be a nonzero normal vector to the discontinuity
boundary $\Sigma_{(i,j)}$ at $x$. The sliding region of
$\Sigma_{(i,j)}$ for system~\eqref{rewritten_system} is defined by $\Sigma_{e,(i,j)}=\left\{x\in\Sigma_{(i,j)}:\bigl(f_i(x)\cdot n_{(i,j)}(x)\bigr)\bigl(f_j(x)\cdot n_{(i,j)}(x)\bigr)<0\right\}$. Thus, the normal components of the adjacent vector fields have opposite signs: both vector fields point toward the boundary (attracting sliding) or away from it (escaping sliding).
On $\Sigma_{e,(i,j)}$, the sliding vector field $f_{e,(i,j)}$
is the unique convex combination of $f_i$ and $f_j$ tangent
to $\Sigma_{(i,j)}$. The corresponding sliding dynamics are
given by $\dot{x}=f_{e,(i,j)}(x)$.
\end{definition}

\begin{definition}
A point $E^*$ in $\mathbb{R}_+^2$ is called a pseudo-equilibrium of system~\eqref{rewritten_system} if $E^*$ is an equilibrium of the sliding vector field on $\Sigma_{e,(i,j)}$. That is, $f_{e,(i,j)}(E^*)=0$ for $E^* \in \Sigma_{e,(i,j)} \subset \Sigma_{(i,j)}$. This is an equilibrium of the sliding dynamics, not an equilibrium of the original smooth subsystems.
\end{definition}

\begin{definition}
Let $\langle \cdot,\cdot \rangle$ indicate the standard inner product, $n_1=(1,0)$ is the normal vector to $\Sigma_{(1,2)}$ and $\Sigma_{(3,4)}$, and $n_2=(0,1)$ is the normal vector to $\Sigma_{(2,3)}$ and $\Sigma_{(4,1)}$. A point $E^*$ in $\mathbb{R}_+^2$ is a tangent point of system~\eqref{rewritten_system} if one of the following holds:
\begin{equation*}
\begin{aligned}
&(1)\ \langle f_1(E^*), n_1\rangle \langle f_2(E^*), n_1\rangle = 0 \quad \text{and } E^* \in \Sigma_{(1,2)};\\
&(2)\ \langle f_2(E^*), n_2\rangle \langle f_3(E^*), n_2\rangle = 0 \quad \text{and } E^* \in \Sigma_{(2,3)};\\
&(3)\ \langle f_3(E^*), n_1\rangle \langle f_4(E^*), n_1\rangle = 0 \quad \text{and } E^* \in \Sigma_{(3,4)};\\
&(4)\ \langle f_4(E^*), n_2\rangle \langle f_1(E^*), n_2\rangle = 0 \quad \text{and } E^* \in \Sigma_{(4,1)}.
\end{aligned}
\end{equation*}
\end{definition}

\begin{definition}
A point $E^*=(N,P)$ in $\mathbb{R}_+^2$ is a boundary equilibrium of system~\eqref{rewritten_system} if there exists $i \in \{1,2,3,4\}$ such that $f_i(E^*)=0$ and $E^*$ lies on one of the discontinuity boundaries $\Sigma_{(i,j)}$ where $(i,j)\in\{(1,2),(2,3),(3,4),(4,1)\}$.
\end{definition}

\subsubsection{Dynamics of the subsystems}
\label{Dynamics_subsystems}
This section investigates the dynamics of the subsystems of~\eqref{rewritten_system} in $\mathbb{R}_+^2$. Throughout this section, we assume that $r,a,b,K,m>0$. Three biologically relevant candidate equilibria of the unforced system~\eqref{rewritten_system} are provided at~\eqref{eqn:equilibria}. Mathematical proofs of stability results are provided in the Supplementary Material~\ref{S:subsystems_dynamics}.

\begin{Lemma}
The equilibria of the smooth subsystem $f_1$ are $E_0$, $E_1$, and $E_2$, where the coexistence equilibrium $E_2$ exists if and only if $abK>m$. The extinction equilibrium $E_0$ is a saddle point. The prey-only equilibrium $E_1$ is locally asymptotically stable if $abK<m$, nonhyperbolic if $abK=m$, and a saddle point if $abK>m$. When $abK>m$, the coexistence equilibrium $E_2$ is locally asymptotically stable.
\label{Lemma_R1}
\end{Lemma}

\begin{theorem}
\label{Theorem_R1}
For the smooth subsystem $f_1$, if $abK>m$, then the positive coexistence equilibrium $E_2$ is globally asymptotically stable in the interior of 
$\mathbb{R}_{+}^2$.
\end{theorem}

\begin{Lemma}
\label{Lemma_R2}
The equilibria of the smooth subsystem $f_2$ are $E_0$ and $\left(\frac{m}{ab},-\frac{r}{a}\right)$. Since the second equilibrium has negative predator density, it is not biologically admissible. $E_0$ is the only biologically relevant equilibrium, and it is locally asymptotically stable.
\end{Lemma}

\begin{theorem}
\label{Theorem_R2}
For the smooth subsystem $f_2$, $E_0$ is globally asymptotically stable in $\mathbb{R}_+^2$. The equilibrium $\left(\frac{m}{ba},-\frac{r}{a}\right)$ is not biologically admissible because it does not belong to $\mathbb{R}_+^2$.
\end{theorem}

\begin{Lemma}
\label{Lemma_R3}
The only equilibrium of the smooth subsystem $f_3$ is the extinction equilibrium $E_0$. Moreover, $E_0$ is locally asymptotically stable.
\end{Lemma}

\begin{theorem}
\label{Theorem_R3}
For the smooth subsystem $f_3$, the extinction equilibrium $E_0$ is globally asymptotically stable in $\mathbb{R}_+^2$.
\end{theorem}

\begin{Lemma}
\label{Lemma_R4}
The equilibria of the smooth subsystem $f_4$ are $E_0$ and $E_1$. The extinction equilibrium $E_0$ is a saddle point, while the prey-only equilibrium $E_1$ is locally asymptotically stable.
\end{Lemma}

\begin{theorem}
\label{Theorem_R4}
The smooth subsystem $f_4$ has two equilibria in $\mathbb{R}_+^2$: $E_0$ and $E_1$. The equilibrium $E_0$ is unstable, while $E_1$ is globally asymptotically stable for all initial conditions with $N(0)>0$ and
$P(0)\ge 0$.
\end{theorem}

The preceding analysis shows that, while a trajectory remains in the interior of a single region and does not encounter a switching boundary $\Sigma_{(i,j)}$, with $ij\in\{(1,2),(2,3),(3,4),(4,1)\}$, its long-term behaviour is governed by the corresponding smooth subsystem. In particular, when $abK>m$, which guarantees that the P-component of $E_2$ is positive, trajectories remaining in the interior of $R_1$ approach the coexistence equilibrium $E_2$; those remaining in the interiors of $R_2$ and $R_3$ approach the extinction equilibrium $E_0$; and those remaining in the interior of $R_4$ approach the prey-only equilibrium $E_1$. However, when a trajectory reaches a switching boundary, the dynamics are no longer determined by a single smooth vector field; crossing, sliding, or escaping motion may occur. This motivates the analysis of the Filippov sliding dynamics in the following section.

\subsubsection{Filippov Sliding Dynamics}
\label{Filippov_Sliding}
Filippov sliding dynamics investigates the behaviour of trajectories on the discontinuity boundaries $\Sigma_{(i,j)}$, with $(i,j)\in\{(1,2),(2,3),(3,4),(4,1)\}$, where the vector field is no longer uniquely determined by a single smooth subsystem. On such boundaries, the adjacent vector fields $f_i$ and $f_j$ may point in the same direction across the boundary, giving crossing dynamics, or they may point in opposite normal directions, giving sliding dynamics. The main conclusions are summarized in the following propositions. Detailed Filippov calculations are provided in the Supplementary Material~\ref{S:sliding_dynamics}. 

\begin{proposition}[Existence and type of sliding segments]
None of the switching manifolds $\Sigma_{(1,2)}$, $\Sigma_{(2,3)}$, $\Sigma_{(3,4)}$, and $\Sigma_{(4,1)}$ admits an attracting sliding segment. Possible sliding motion is only of escaping type. More precisely, escaping sliding regions are
\begin{subequations}\label{eq:sigma-e}
\renewcommand{\theequation}{\theparentequation\Alph{equation}}
\begin{align}
\Sigma_{e,(1,2)}&= \left\{(N,P): N=N_a,\; P_a<P<\frac{r}{a}
\left(1-\frac{N_a}{K}\right)\right\},
\label{eq:sigma-e-1-2} \\[4pt]
\Sigma_{e,(2,3)}&= \left\{(N,P): \frac{m}{ba}<N<N_a,\; P=P_a\right\},
\label{eq:sigma-e-2-3} \\[4pt]
\Sigma_{e,(3,4)}&= \left\{(N,P): N=N_a,\; 0<P<P_a,\; N_a<K\right\},
\label{eq:sigma-e-3-4} \\[4pt]
\Sigma_{e,(4,1)}&= \left\{(N,P): N>\max\left\{N_a,\frac{m}{ba}\right\},\; P=P_a\right\}.
\label{eq:sigma-e-4-1}
\end{align}

\end{subequations}
Each set is understood to be empty when its defining inequalities cannot be satisfied.
\end{proposition} 

\begin{proposition}[Pseudo-equilibria on the switching manifolds]\label{Pseudo}
The escaping sliding dynamics on $\Sigma_{(2,3)}$ and $\Sigma_{(3,4)}$ admits no positive pseudo-equilibrium. The segment $\Sigma_{(1,2)}$ admits no isolated positive pseudo-equilibrium, except in the degenerate case $N_a=m/(ba)$, where every point of $\Sigma_{e,(1,2)}$ is an equilibrium of the escaping sliding dynamics. The only switching manifold that can support isolated pseudo-equilibria is $\Sigma_{(4,1)}$, with the candidate pseudo-equilibria 
\begin{equation*}
E_{p4}^{\pm} = \left(\frac{K\pm\sqrt{K^2-\frac{4KmP_a}{br}}}{2}, P_a \right), \quad \text{provided that} \quad K^2-\frac{4KmP_a}{br}\geq 0 \quad \text{and} \quad \frac{K\pm\sqrt{K^2-\frac{4KmP_a}{br}}}{2}>\max\left\{N_a,\frac{m}{ba}\right\}.
\end{equation*}
\end{proposition}

\subsubsection{Bifurcation diagram with respect to \texorpdfstring{$r$}{r}}\label{Bifurcation_diagram}
We now examine the bifurcation structure with respect to the intrinsic prey growth rate $r$. In particular, we focus on the transition that occurs when the coexistence equilibrium $E_2$ of the smooth subsystem $f_1$ in region $R_1$ reaches the switching boundary $N=N_a$ or $P=P_a$. Using the parameter values in~\eqref{eqn:parameters}, we obtain $(N^*,P^*)=(1.2784,0.7154r)$. For $E_2(r)$ to be virtual in $R_1$, it must satisfy $N^*>N_a$ and $P^*(r)>P_a$. The first condition is automatically satisfied since $N^*=1.2784>1=N_a$. The second condition gives the critical value
\begin{equation}
    r_c=\frac{aP_a}{1-\frac{m}{baK}}\approx0.978484.
    \label{eq:critical_r_c}
\end{equation}
$E_2(r)$ is not virtual for $r<r_c$, and lies on the switching boundary $P=P_a$ at $r=r_c$. Since $E_2$ of the subsystem $f_1$ is locally stable whenever it is virtual, this transition corresponds to the emergence of a stable admissible coexistence equilibrium at $r=r_c$.

The bifurcation diagram of system~\eqref{rewritten_system} is shown in Figure~\ref{fig:bifurcation_phase_portraits_rc} (A,B). The stable boundary equilibria $E_0=(0,0)$ and $E_1=(K,0)$ are plotted together with the stable coexistence branch $E_2(r)$. The pseudo-equilibria on the escaping sliding segment $\Sigma_{e,(4,1)}$ are shown as dashed curves, indicating their instability. The vertical line $r=r_c$ marks the point where $E_2(r)$ crosses the switching boundary and becomes virtual in $R_1$. Panels~(C) and~(D) compare the phase portraits slightly below and above this critical value $r_c$. With initial condition near $E_2$, if $r<r_c$, the point $E_2(r)$ lies below the switching threshold $P=P_a$ and is therefore not a virtual equilibrium in $R_1$; if $r>r_c$, the same branch has entered $R_1$, and the trajectory initialised at $E_2(r)$ remains at the stable coexistence equilibrium.

\begin{figure}[ht]
    \centering
    \includegraphics[width=\textwidth]{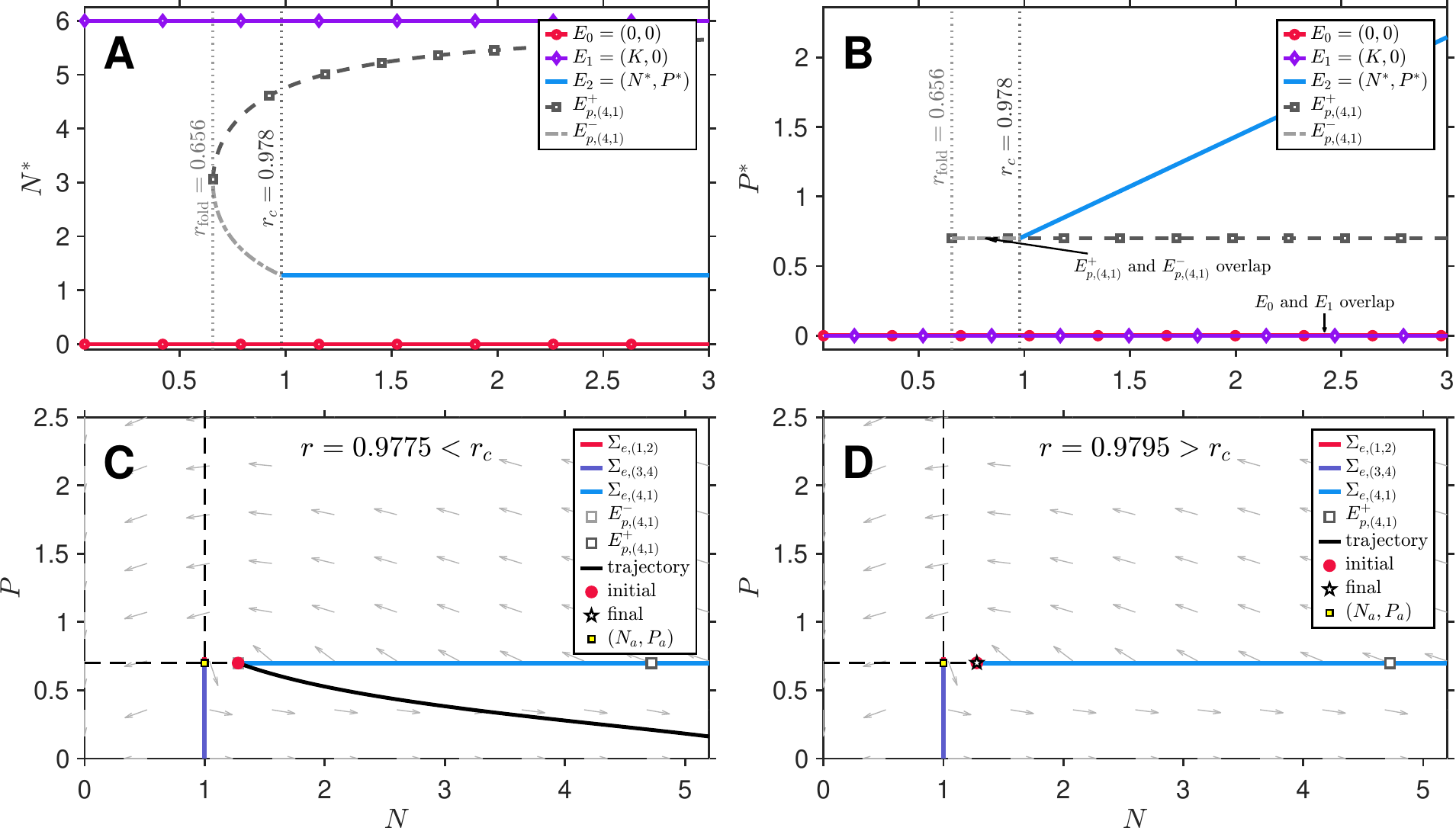}
    \caption{Analytical bifurcation diagram and phase portraits near the critical value $r_c$. (A,B) show the equilibrium and pseudo-equilibrium branches in the $N$- and $P$-components, respectively. The solid branches represent stable equilibria, while dashed branches represent unstable pseudo-equilibria on the escaping sliding segment $\Sigma_{(4,1)}$. (C,D) are phase portraits for $r=r_c-0.001<r_c$ and $r=r_c+0.001>r_c$, respectively, with the initial condition chosen exactly at $E_2(r)$.}
    \label{fig:bifurcation_phase_portraits_rc}
\end{figure}

\subsection{Filippov predator-prey model with Allee effect and a periodically varying growth rate}
The non-autonomous piecewise-smooth system~\eqref{predator_prey_model_varying_r} can be written as
\begin{equation}
\begin{pmatrix}
\dfrac{dN}{dt} \\[6pt]
\dfrac{dP}{dt}
\end{pmatrix}=
\begin{cases}
f_1(t,(N,P)), & \text{for } (N,P) \in R_1 = \{(N,P)\in\mathbb{R}^2_+:N>N_a,\ P>P_a\},\\
f_2(t,(N,P)), & \text{for } (N,P) \in R_2 = \{(N,P)\in\mathbb{R}^2_+:N<N_a,\ P>P_a\},\\
f_3(t,(N,P)), & \text{for } (N,P) \in R_3 = \{(N,P)\in\mathbb{R}^2_+:N<N_a,\ P<P_a\},\\
f_4(t,(N,P)), & \text{for } (N,P) \in R_4 = \{(N,P)\in\mathbb{R}^2_+:N>N_a,\ P<P_a\},
\end{cases}
\label{rewritten_system_nonautonomous}
\end{equation}
where the corresponding time-periodic vector fields are
\begin{equation*}
\begin{aligned}
f_1(t,(N,P))&=
\binom{\rho(t)N\left(1-\frac{N}{K}\right)-aNP}{baNP-mP},
&
f_2(t,(N,P))&=
\binom{-\rho(t)N-aNP}{baNP-mP},
\\[0.8em]
f_3(t,(N,P))&=
\binom{-\rho(t)N}{-mP},
&
f_4(t,(N,P))&=
\binom{\rho(t)N\left(1-\frac{N}{K}\right)}{-mP}.
\end{aligned} 
\end{equation*}

\subsubsection{Criterion of frequency-induced tipping}
\label{F_tipping_criterion}
To identify a transition as frequency-induced, we first exclude transitions that would occur even in the absence of periodic forcing. Let $x^0(t)=(N^0(t),P^0(t))$ denote the solution of the unforced system with $\rho(t)\equiv r_0$ and the same initial condition $(N_0,P_0)\in R_1$. We require this unforced trajectory to remain in the basin of attraction of the coexistence equilibrium $E_2$ and to have no admissible threshold crossing. Equivalently, according to equations~\eqref{eq:sigma-e-1-2} and~\eqref{eq:sigma-e-4-1}, we assume $t_c^0=\infty$ where $t_c^0$ denotes the first admissible crossing time for the unforced system,
\begin{equation*}
t_c^0 =
\inf\left\{t>0:
\begin{array}{l}
N^0(t)=N_a
\ \text{and}\
P^0(t)>\dfrac{r_0}{a}\left(1-\dfrac{N_a}{K}\right), \quad \text{or} \quad P^0(t)=P_a
\ \text{and}\
N_a<N^0(t)<\dfrac{m}{ba}
\end{array}
\right\}. 
\end{equation*}
This baseline condition ensures that any subsequent transition under periodic forcing is not caused by the initial condition being outside the basin of attraction of $E_2$, but by the temporal structure of the forcing itself.

For a fixed forcing amplitude $A$, let $(N_\tau(t),P_\tau(t))$ denote the flow of subsystem $f_1$ under forcing period $\tau$ with the initial condition $(N_0,P_0)$. Define the values (intervals) of forcing periods $\tau$ for which the historical minimum of $N$ and $P$ fall below the Allee thresholds $N_a$ and $P_a$
\begin{equation*}
\mathcal{T}_N(A)=\left\{\tau>0:\min_{t\in[0,\infty)}N_\tau(t)\leq N_a\right\},
\qquad
\mathcal{T}_P(A)=\left\{\tau>0:\min_{t\in[0,\infty)}P_\tau(t)\leq P_a\right\}.
\end{equation*}
A necessary condition for frequency-induced tipping is therefore
\begin{equation}
    \exists\tau_{\mathrm{tip}}\in \mathcal{T}_N(A)\cup\mathcal{T}_P(A).
\end{equation}
This condition identifies forcing periods for which resonance amplifies the population oscillations sufficiently for at least one threshold to be reached. Threshold reaching alone, however, does not guarantee tipping. The Filippov analysis in section~\ref{Filippov_Sliding} shows that all sliding segments of the switching manifolds are escaping. When a trajectory reaches a threshold on an escaping segment, it is repelled from the switching manifold and returns to the original region $R_1$, rather than crossing into a different region. For the time-dependent system, the escaping portion of $\Sigma_{(1,2)}$ is obtained by replacing $r$ with $\rho(t)$:
\begin{equation*}
    \Sigma_{e,(1,2)}(t) = \left\{ (N,P):N=N_a,\;P_a<P<\frac{\rho(t)}{a}\left(1-\frac{N_a}{K}\right)
\right\}.
\end{equation*}
Consequently, crossing from $R_1$ into $R_2$ through the prey threshold $N=N_a$ is possible only when the trajectory reaches $\Sigma_{(1,2)}$ outside $\Sigma_{e,(1,2)}(t)$. In particular, this requires a crossing time $t_c$ such that
\begin{equation*}
    N(t_c)=N_a, \qquad P(t_c)>\frac{\rho(t_c)}{a}\left(1-\frac{N_a}{K}\right).
\end{equation*}
Similarly, the escaping portion of $\Sigma_{(4,1)}$ is independent of $\rho(t)$ and is given by
\begin{equation*}
    \Sigma_{e,(4,1)} = \left\{ (N,P): N>\max\left\{N_a,\frac{m}{ba}\right\},\;P=P_a\right\}.
\end{equation*}
A transition from $R_1$ into $R_4$ through the predator threshold $P=P_a$ can occur only if there exists a crossing time $t_c$ such that
\begin{equation*}
  P(t_c)=P_a, \qquad N_a<N(t_c)<\frac{m}{ba}.  
\end{equation*}

\section*{Author contributions}
S.G.: conceptualization, formal analysis, investigation, methodology, software, visualization, writing-original draft, writing-review and editing; S.P.: methodology, writing-review and editing; H.W.: funding acquisition, writing-review and editing, supervision; M.A.L: funding acquisition, writing-review and editing, supervision.

\section*{Declaration of AI use}
ChatGPT 5 was used to assist in running the simulation code on Digital Research Alliance of Canada (formerly Compute Canada). The resulting simulation data were only used to generate Figure~\ref{fig:figure5}. All AI-assisted work has been carefully checked by the authors, who take full responsibility for the results presented in this study. AI was not used for conceptualization, mathematical analysis, numerical simulations, manuscript writing, or figure creation.

\section*{Conflict of interest declaration}
We declare we have no competing interests. 

\section*{Code accessibility}
Code to reproduce the simulations and figures is available at \url{https://doi.org/10.5281/zenodo.23653919}.
\section*{Funding}
Shan Gao acknowledges funding from the Maud Menten Institute Research Collaboration Accelerator program. Hao Wang's research was partially supported by the Natural Sciences and Engineering Research Council of Canada (Discovery Grant RGPIN-2025-05734) and a Canada Research Chair. Mark A. Lewis gratefully acknowledges the Kennedy Research Chair in Mathematical Biology from the University of Victoria, and the NSERC Discovery Grant program (RGPIN-2024-04792). Funding for this research has also been provided through grants to the TRIA-FoR Project to Mark A. Lewis from Genome Canada (project number 18202) and the Government of Alberta through Genome Alberta (grant number L20TF).

\section*{Acknowledgments}
We would like to thank Lewis Research Group and Interdisciplinary Lab for Mathematical Ecology \& Epidemiology members for providing valuable feedback. 

%%%%%%%%%%%%%%%%%%%%%%%%%%%%%%%%%%%%%%%%%%%%%%%%%%%%%%%%%%%%%%%%%%%%%%%%%%%%%%%%
\newpage
\section*{Supplementary Information}
\addcontentsline{toc}{section}{Supplementary information}

\appendix
\renewcommand{\thesection}{\Alph{section}}
\renewcommand{\thesubsection}{\thesection\arabic{subsection}}
\renewcommand{\thesubsubsection}{\thesubsection.\arabic{subsubsection}}
\setcounter{secnumdepth}{3}

% Supplementary figure/table numbering
\setcounter{figure}{0}
\setcounter{table}{0}
\renewcommand{\thefigure}{S\arabic{figure}}
\renewcommand{\thetable}{S\arabic{table}}
%%%%%%%%%%%%%%%%%%%%%%%%%%%%%%%%%%%%%%%%%%%%%%%%%%%%%%%%%%%%%%%%%%%%%%%%%%%%%%%%

\section{Rescaling the system}
\label{sec: app-resc}

We rewrite the system \eqref{predator_prey_model} in dimensionless variables. For that, let $[N]$ represent the dimension of $N$, and similarly for other parameters and variables of the model. For example, we will use $[K]=[N]$ to state that the dimensions of carrying capacity, $[K]$, equal the dimensions of prey population size, $[N]$. Note that $[a]=[1/(tP)]$ and $[b]=[P/N]$, so that $[bK]=[P]$ and $[abK]=[1/t]$. We therefore rescale our variables as: $n=N/K$, $p=P/(bK)$, and $\tau = abKt$, leading to

\begin{equation}
\begin{pmatrix}
\frac{\text{d}n}{\text{d}t} \\[6pt]
\frac{\text{d}p}{\text{d}t}
\end{pmatrix}
=
\begin{pmatrix}
\tilde{r}n(1-n)\theta(n-n_a)-\tilde{r}n\theta(n_a-n)-np\theta(p-p_a) \\[6pt]
np\theta(p-p_a)-\tilde{m}p
\end{pmatrix},
\label{predator_prey_model_rescaled}
\end{equation}
where $n_a=N_a/K$, $p_a=P_a/(bK)$, $\tilde{r}=r/(abK)$, and $\tilde{m}=m/(abK)$.

The fixed points of model \eqref{predator_prey_model_rescaled} are $(n,p)=(0,0)$ (extinction), $(n,p)=(1,0)$ (prey only) and $(n,p)=(n^*,p^*)=(\tilde{m},\tilde{r}(1-\tilde{m}))$ (coexistence). The eigenvalues of the Jacobian matrix, $\tilde{\lambda}_\pm$, about the coexistence fixed point are:

\begin{equation}
    \tilde{\lambda}_\pm=-\frac{\tilde{r} \tilde{m}}{2}\pm i\tilde\omega,\quad\text{where}\quad \tilde\omega = \sqrt{\tilde{r}\tilde{m}(1-\tilde{m})-\frac{\tilde{r}^2\tilde{m}^2}{4}}.
\end{equation}
Note that, under the assumption that the coexistence fixed point is positive, i.e., $\tilde{r}>0$ and $0<\tilde{m}<1$, the fixed point is always stable. Furthermore, if $\tilde{r}<4(1-\tilde m)/\tilde m$, then the coexistence fixed point is a stable focus, while if the opposite holds, then the the coexistence fixed point is a stable node. In case of equality, $\tilde{r}=4(1-\tilde m)/\tilde m$, the Jacobian has a repeated negative eigenvalue.

Note that the condition $\tilde{r}<4(1-\tilde m)/\tilde m$ imply that the larger is predator mortality, the smaller the prey growth rate must be in order for oscillations to appear. Biologically, this creates the possibility of a feedback between both timescales, fast predator mortality dominates when prey density is low, reducing predator population and allowing prey population growth to take place, despite being slow. Once prey becomes abundantly available again, predator population can grow, and the cycle repeats. However, since prey growth is also self-limited, these cycles are damped over time. Importantly, these dynamics depend only implicitly on $a,b$ and $K$, leading for the same qualitative outcomes for different predator and prey species' compositions. 

\section{A continuous approximation of the Filippov system}\label{continuous_approximation_analysis}
Replacing the Heaviside step function in the predator-prey system~\eqref{predator_prey_model} with the logistic function~\eqref{logistic_fun} gives
\begin{equation}
\begin{pmatrix}
\dfrac{dN}{dt} \\[6pt]
\dfrac{dP}{dt}
\end{pmatrix}
=
\begin{pmatrix}
rN\left(1-\dfrac{N}{K}\right)\gamma(N-N_a)
-rN\gamma(N_a-N)
-aNP\gamma(P-P_a) \\[8pt]
baNP\gamma(P-P_a)-mP
\end{pmatrix}.
\label{predator_prey_model_continuous}
\end{equation}

\subsection{Theoretical analysis}
System \eqref{predator_prey_model_continuous} can be written as
\begin{equation*}
    \frac{dN}{dt}=N F(N,P),
    \qquad
    \frac{dP}{dt}=P G(N,P),
\end{equation*}
where
\begin{equation*}
    F(N,P)=r\left[\left(2-\frac{N}{K}\right)\gamma(N-N_a)-1\right]-aP\gamma(P-P_a),
\end{equation*}
and
\begin{equation*}
    G(N,P)=baN\gamma(P-P_a)-m.
\end{equation*}

The derivatives of the logistic functions are
\begin{equation*}
    \frac{d\gamma(N-N_a)}{dN}=\alpha \gamma(N-N_a)\gamma(N_a-N),
    \qquad
    \frac{d\gamma(P-P_a)}{dP}=\alpha \gamma(P-P_a)\gamma(P_a-P).
\end{equation*}
Hence, the Jacobian matrix is
\begin{equation*}
J(N,P)=
\begin{pmatrix}
F(N,P)+NF_N & N F_P \\[4pt]
P G_N & G(N,P)+P G_P
\end{pmatrix},
\label{continuous_jacobian}
\end{equation*}
where
\begin{align*}
F_N&=r\left[-\frac{\gamma(N-N_a)}{K}+\alpha\left(2-\frac{N}{K}\right) \gamma(N-N_a)\gamma(N_a-N))\right],\\
F_P&=-a\left[\gamma(P-P_a)+\alpha P\gamma(P-P_a)\gamma(P_a-P))\right],\\
G_N&=ba\gamma(P-P_a),\\ 
G_P&=baN\alpha \gamma(P-P_a)\gamma(P_a-P).
\end{align*}

\subsubsection*{Extinction equilibrium}
The extinction equilibrium is $E_0=(0,0)$, with Jacobian:
\begin{equation*}
    J(E_0)=
    \begin{pmatrix}
        r\left[2\gamma(-N_a)-1\right] & 0\\
        0 & -m
    \end{pmatrix}.
\end{equation*}
The eigenvalues are
\begin{equation*}
    \lambda_1=r\left[2\gamma(-N_a)-1\right],
    \qquad
    \lambda_2=-m.
\end{equation*}
For $r,m,\alpha,N_a>0$, we have $\gamma(-N_a)<1/2$, implying that both eigenvalues are negative. $E_0$ is locally asymptotically stable.

\subsubsection*{Prey-only equilibria}
Prey-only equilibria have the form $E_1=(\bar N,0)$, where $\bar N>0$ satisfies
\begin{equation*}
    \left(2-\frac{\bar N}{K}\right)
    \gamma(\bar N-N_a)=1.
    \label{prey_only_equation}
\end{equation*}
Equivalently, we have $1-\frac{\bar N}{K}=e^{-\alpha(\bar N-N_a)}$, which gives $N_a<\bar N<K$ provided $N_a<K$. Define $h(N)=1-\frac{\bar N}{K}-e^{-\alpha(\bar N-N_a)}$ and we have $h''(N)=-\alpha^2e^{-\alpha(N-N_a)}<0$, indicating that the $h(N)$ is strictly concave. We further have $
h(0)=1-e^{\alpha N_a}<0$ and
$h(K)=-e^{-\alpha(K-N_a)}<0$. Hence, if its maximum occurs at a positive value of $N$, we have (a) A positive maximum gives two positive crossings; (b) A zero maximum gives one tangential, double root; or (c) A negative maximum gives no roots. For such an equilibrium, the Jacobian is triangular, and its eigenvalues are
\begin{equation*}
\lambda_N=\bar N r\left[-\frac{\gamma(\bar N-N_a)}{K}+\alpha\left(1-\gamma(\bar N-N_a)\right)\right],
\end{equation*}
and
\begin{equation*}
    \lambda_P=ba\bar N\gamma(-P_a)-m.
\end{equation*}
Thus, a prey-only equilibrium is locally asymptotically stable if both $\lambda_N<0$ and $\lambda_P<0$.

\subsubsection*{Coexistence equilibria}
Coexistence equilibria have the form
\begin{equation*}
    E_2=(N^*,P^*),
    \qquad N^*>0,\quad P^*>0.
\end{equation*}
From the predator-nullcline condition,
\begin{equation*}
    baN^*\gamma(P^*-P_a)=m,
\end{equation*}
we obtain
\begin{equation}
    P^*=P_a+\frac{1}{\alpha}\log\left(\frac{m}{baN^*-m}\right),
    \qquad N^*>\frac{m}{ba}.
    \label{coexistence_Pstar}
\end{equation}
Requiring $P^*>0$ further gives
\begin{equation*}
    \frac{m}{ba} < N^* < \frac{m}{ba} \left( 1 + e^{\alpha P_a} \right) = \frac{m}{ba \gamma(-P_a)}.
\end{equation*}
Substitution into the prey-nullcline condition yields the scalar equation
\begin{equation}
r\left[\left(2-\frac{N^*}{K}\right)\gamma(N^*-N_a)-1\right]=\frac{m}{bN^*}\left[P_a+\frac{1}{\alpha}\log\left(\frac{m}{baN^*-m}\right)\right].
\label{coexistence_Nstar}
\end{equation}
Hence, coexistence equilibria are obtained by solving~\eqref{coexistence_Nstar} numerically for $N^*$ and then using~\eqref{coexistence_Pstar} to determine $P^*$.

At a coexistence equilibrium, $F(N^*,P^*)=G(N^*,P^*)=0$, and thus
\begin{equation*}
J(E_2)=
\begin{pmatrix}
N^*F_N & N^*F_P\\
P^*G_N & P^*G_P
\end{pmatrix}_{(N^*,P^*)}.
\end{equation*}
The coexistence equilibrium is locally asymptotically stable if
\begin{equation*}
    \operatorname{tr}\bigl(J(E_2)\bigr)<0,
    \qquad
    \det\bigl(J(E_2)\bigr)>0.
\end{equation*}

\subsubsection*{Equilibria and stability for the chosen parameter set}
For the following parameter set:
\begin{equation}\label{parameter_continuous}
r=1.2,\quad K=6.0,\quad a=1.1,\quad b=0.32,\quad
m=0.45,\quad N_a=1.0,\quad P_a=0.7,\quad \alpha=30,
\end{equation}
the smooth system admits five equilibria: one extinction equilibrium, two prey-only equilibria, and two coexistence equilibria. Their locations, eigenvalues, and corresponding local stability properties are summarized in the table below.
\begin{table}[ht]
\centering
\begin{tabular}{lcl}
\toprule
\text{Equilibrium} & \text{Eigenvalues} & \text{Local stability}\\
\midrule
$E_0=(0,0)$
& $-1.200000,\,-0.450000$
& stable node\\[3pt]
$E_1^{(1)}\approx(1.006118,0)$
& $16.342938,\,-0.450000$
& saddle\\[3pt]
$E_1^{(2)}\approx(6.000000,0)$
& $-1.200000,\,-0.450000$
& stable node\\[3pt]
$E_2^{(1)}\approx(1.288070,0.862842)$
& $-0.077797\pm0.692044\,\mathrm{i}$
& stable focus\\[3pt]
$E_2^{(2)}\approx(5.081743,0.663658)$
& $-0.842241,\,6.531372$
& saddle\\
\bottomrule
\end{tabular}
\end{table}

Therefore, for this parameter set, the smooth system has three locally asymptotically stable equilibria: the extinction equilibrium $E_0$, the high-prey prey-only equilibrium $E_1^{(2)}$, and the coexistence equilibrium $E_2^{(1)}$. The natural period of $E_2^{(1)}$ is $\frac{2\pi}{0.692044}\approx 9.08$.

\section{Proofs of the subsystem stability results.}
\label{S:subsystems_dynamics}
\subsection{Dynamics of subsystem 1}
\label{S:subsystems_dynamics1}
\begin{proof}[Proof of Lemma~\ref{Lemma_R1}]
In region $R_1$, i.e. $N>N_a,\ P>P_a$, the system~\ref{predator_prey_model} becomes
\begin{equation}
\left\{
\begin{aligned}
\dfrac{dN}{dt} &= rN\left(1-\frac{N}{K}\right)-aNP,\\
\dfrac{dP}{dt} &= P(baN-m),
\end{aligned}
\right.
\label{subsystem1}
\end{equation}
Let $\frac{dN}{dt}$=$\frac{dP}{dt}$=0, we can immediately have the extinction equilibrium $E_0=(0,0)$ and prey-only equilibrium $E_1=(K,0)$. Also, the coexistence equilibrium is $E_2=\left(\frac{m}{ab},\ \frac{r}{a}\left(1-\frac{m}{baK}\right)\right)$, provided $abK>m$.

The Jacobian matrix is
\begin{equation*}
    J(N,P)=\begin{pmatrix}
    r-\frac{2rN}{K}-aP & -aN\\
    baP & baN-m
    \end{pmatrix}.
\end{equation*}

Therefore, we can compute $J(E_0)$, $J(E_1)$, and $J(E_2)$, as shown below.
\begin{equation*}
J(E_0)=
\begin{pmatrix}
r & 0\\
0 & -m
\end{pmatrix},
\qquad
J(E_1)=
\begin{pmatrix}
-r & -aK\\
0 & baK-m
\end{pmatrix},
\qquad
J(E_2)=
\begin{pmatrix}
-\dfrac{rm}{baK} & -\dfrac{m}{b}\\[10pt]
br\left(1-\dfrac{m}{baK}\right) & 0
\end{pmatrix}.
\end{equation*}
At $E_0=(0,0)$, the eigenvalues are $\lambda_1=r>0$ and $\lambda_2=-m<0$. Therefore, $E_0=(0,0)$ is a saddle point. At $E_1=(K,0)$, the eigenvalues are $\lambda_1=-r<0$ and $\lambda_2=baK-m$. The local stability of $E_1$ depends on the sign of $baK-m$:
\begin{equation*}
E_1=(K,0)
\begin{cases}
\text{is locally asymptotically stable}, & baK<m,\\[2mm]
\text{is nonhyperbolic}, & baK=m,\\[2mm]
\text{is a saddle point}, & baK>m.
\end{cases}
\end{equation*}
At $E_2$, the trace and determinant of $J(E_2)$ are $\operatorname{tr}J(E_2)=-\frac{rm}{baK}<0$, and $\det J(E_2)=mr\left(1-\frac{m}{baK}\right)$. Whenever the coexistence equilibrium exists, we have $\frac{m}{baK}<1$, so $\det J(E_2)>0$. Therefore, $E_2$ is locally asymptotically stable if $baK>m$. Additionally, we can compute the eigenvalues
\begin{equation}
\lambda_{\pm} = -\frac{rm}{2abK} \pm i\omega, \qquad \text{where }\omega = \sqrt{mr\left(1-\frac{m}{baK}\right)-\frac{r^2m^2}{4a^2b^2K^2}}.
\end{equation}
\end{proof}

\begin{proof}[Proof of Theorem~\ref{Theorem_R1}]
First, the nonnegative quadrant $\mathbb{R}_+^2$ is positively invariant, because we have
\begin{equation}
\begin{aligned}
N(t)&=N(0)\exp\left\{\int_0^t
\left[r\left(1-\frac{N(s)}{K}\right)-aP(s)\right]\,ds\right\}>0,\\
P(t)&=P(0)\exp\left\{\int_0^t (baN(s)-m)\,ds\right\}>0,
\end{aligned}
\end{equation}
for all $t>0$ when $N(0)>0$ and $P(0)>0$. Moreover, the coordinate axes are invariant. Hence $\mathbb{R}_+^2$ is positively invariant.

Assume $baK>m$, then the coexistence equilibrium $E_2$ lies in $\mathbb{R}_{+}^2$. Define the Lyapunov function
\begin{equation*}
   V(N,P)=b\left(N-N^*-N^*\ln\frac{N}{N^*}\right)+\left(P-P^*-P^*\ln\frac{P}{P^*}\right). 
\end{equation*}

Since $x-1-\ln x\ge 0$ for all $x>0$, with equality if and only if $x=1$, it follows that $V(N,P)\ge 0$, and $V(N,P)=0$ if and only if $(N,P)=(N^*,P^*)$. Differentiating $V$ along solutions of system~\eqref{subsystem1} gives
\begin{equation*}
    \begin{aligned}
    \dot V &= b\left(1-\frac{N^*}{N}\right)\dot N + \left(1-\frac{P^*}{P}\right)\dot P \\[2mm]
    &=b(N-N^*)\left[r\left(1-\frac{N}{K}\right)-aP\right] + (P-P^*)(baN-m).
\end{aligned}
\end{equation*}
Using the equilibrium relations
\begin{equation*}
    r\left(1-\frac{N^*}{K}\right)-aP^*=0,\qquad baN^*-m=0,
\end{equation*}
we have
\begin{equation}
\begin{aligned}
r\left(1-\frac{N}{K}\right)-aP&=-\frac{r}{K}(N-N^*)-a(P-P^*),\\
baN-m&=ba(N-N^*).
\end{aligned}
\end{equation}
Therefore, $V$ is nonincreasing along solutions since
\begin{equation*}
    \begin{aligned}
    \dot V = b(N-N^*)\left[-\frac{r}{K}(N-N^*)-a(P-P^*)\right]+(P-P^*)ba(N-N^*) = -\frac{br}{K}(N-N^*)^2\le 0.
\end{aligned}
\end{equation*}
Moreover, the sublevel sets of $V$ are compact in $\mathbb{R}_{+}^2$, since $V(N,P)\to\infty$ as $N\to0^+$, $P\to0^+$, $N\to\infty$, or $P\to\infty$. Each positive solution is contained in a compact positively invariant subset of $\mathbb{R}_{+}^2$. The set where $\dot V=0$ is given by $N=N^*$. On this set,
\begin{equation*}
    \dot N=N^*\left[r\left(1-\frac{N^*}{K}\right)-aP\right] = aN^*(P^*-P).
\end{equation*}
Thus a trajectory can remain in $\{N=N^*\}$ only if $P=P^*$. Therefore, the largest invariant set contained in $\{\dot V=0\}$ is $\{E_2\}$. By LaSalle's invariance principle~\cite{lasalle1960some}, every solution with initial condition $(N(0),P(0))\in\mathbb{R}_{+}^2$ converges to $E_2$. Hence $E_2$ is globally asymptotically stable in $\mathbb{R}_{+}^2$.
\end{proof}

\subsection{Dynamics of subsystem 2}
\begin{proof}[Proof of Lemma~\ref{Lemma_R2}]
In region $R_2$, i.e. $N<N_a,\ P>P_a$, the system~\eqref{predator_prey_model} becomes
\begin{equation}
\left\{
\begin{aligned}
\dfrac{dN}{dt} &= -rN-aNP,\\
\dfrac{dP}{dt} &= P(baN-m),
\end{aligned}
\right.
\label{subsystem2}
\end{equation}
The equilibria are $E_0=(0,0)$ and $E_1=\left(\frac{m}{ba},-\frac{r}{a}\right)$. Because $E_1$ has a negative $P$-component, it is not biologically meaningful in the first quadrant. Thus, we consider only the equilibrium $E_0=(0,0)$. The Jacobian matrix is
\begin{equation*}
    J(N,P)=\begin{pmatrix}
    -r-aP & -aN\\
    baP & baN-m
\end{pmatrix}.
\end{equation*}
Evaluating the Jacobian at the equilibria gives
\begin{equation*}
    J(E_0)=\begin{pmatrix}
    -r & 0\\
    0 & -m
\end{pmatrix}.
\end{equation*}
$E_0=(0,0)$ is locally asymptotically stable since the eigenvalues are $\lambda_1=-r<0$ and $\lambda_2=-m<0$.   
\end{proof}

\begin{proof}[Proof of Theorem~\ref{Theorem_R2}]
The nonnegative quadrant $\mathbb{R}_+^2$ is positively invariant since both equations contain the factors $N$ and $P$. Moreover, because $P(t)\ge 0$, we have
\begin{equation*}
\dot N=-N(r+aP)\le -rN.
\end{equation*}
Hence $0\le N(t)\le N(0)e^{-rt}$, and therefore $N(t)\to 0$ as
$t\to\infty$. Next,
\begin{equation*}
P(t)=P(0)\exp\left\{\int_0^t(baN(s)-m)\,ds\right\}.
\end{equation*}
Since $N(t)\le N(0)e^{-rt}$, we have
\begin{equation*}
\int_0^t N(s)\,ds\le \frac{N(0)}{r}.
\end{equation*}
Thus,
\begin{equation*}
P(t)\le P(0)\exp\left\{\frac{baN(0)}{r}-mt\right\},
\end{equation*}
which implies $P(t)\to 0$ as $t\to\infty$. Therefore every solution in $\mathbb{R}_+^2$ converges to $E_0=(0,0)$. Finally, since $E_0$ is locally asymptotically stable, all solutions in $\mathbb{R}_+^2$ converge to $E_0$, it is globally asymptotically stable in $\mathbb{R}_+^2$.
\end{proof}

\subsection{Dynamics of subsystem 3}
\begin{proof}[Proof of Lemma~\ref{Lemma_R3}]
In region $R_3$, i.e. $N<N_a,\ P<P_a$, the system~\ref{predator_prey_model} becomes
\begin{equation}
\left\{
\begin{aligned}
\dfrac{dN}{dt} &= -rN,\\
\dfrac{dP}{dt} &= -mP,
\end{aligned}
\right.
\label{subsystem3}
\end{equation}
In this case, we only have  extinction equilibrium $E_0=(0,0)$. Near $E_0$, we are in the region $N<N_a$, $P<P_a$, so the Jacobian matrix at $E_0$ is provided below, with eigenvalues $\lambda_1=-r<0$ and $\lambda_2=-m<0$. Since both are negative, $E_0$ is locally asymptotically stable for all $r>0$.
\begin{equation}
J(E_0)=
\begin{pmatrix}
-r & 0\\
0 & -m
\end{pmatrix},
\end{equation}
\end{proof}
\begin{proof}[Proof of Theorem~\ref{Theorem_R3}]
The nonnegative quadrant $\mathbb{R}_+^2$ is positively invariant since the coordinate axes are invariant. Solving the system explicitly gives
\begin{equation*}
    N(t)=N(0)e^{-rt},\qquad P(t)=P(0)e^{-mt}.
\end{equation*}
Thus, for any initial condition $(N(0),P(0))\in\mathbb{R}_+^2$, we have
$N(t)\ge 0$ and $P(t)\ge 0$ for all $t\ge 0$, and
\begin{equation*}
\lim_{t\to\infty}N(t)=0,
\qquad
\end{equation*}
Therefore every solution in $\mathbb{R}_+^2$ converges to $E_0=(0,0)$. Moreover, since $E_0$ is locally asymptotically stable, $E_0=(0,0)$ is globally asymptotically stable in $\mathbb{R}_+^2$.
\end{proof}

\subsection{Dynamics of subsystem 4}
\begin{proof}[Proof of Lemma~\ref{Lemma_R4}]
In region $R_4$, i.e. $N>N_a,\ P<P_a$, the system~\ref{predator_prey_model} becomes
\begin{equation}
\left\{
\begin{aligned}
\dfrac{dN}{dt} &= rN\left(1-\frac{N}{K}\right),\\
\dfrac{dP}{dt} &= -mP,
\end{aligned}
\right.
\label{subsystem4}
\end{equation}
The equilibria are the extinction equilibrium $E_0=(0,0)$ and the prey-only equilibrium $E_1=(K,0)$. The Jacobian matrix is
\begin{equation*}
J(N,P)=
\begin{pmatrix}
r-\dfrac{2rN}{K} & 0\\
0 & -m
\end{pmatrix}.
\end{equation*}
Evaluating the Jacobian at the equilibria gives
\begin{equation*}
J(E_0)=
\begin{pmatrix}
r & 0\\
0 & -m
\end{pmatrix},
\qquad
J(E_1)=
\begin{pmatrix}
-r & 0\\
0 & -m
\end{pmatrix}.
\end{equation*}
At $E_0=(0,0)$, the eigenvalues are $\lambda_1=r>0$ and $\lambda_2=-m<0$. $E_0=(0,0)$ is a saddle point. At $E_1=(K,0)$, the eigenvalues are $\lambda_1=-r<0$ and $\lambda_2=-m<0$. $E_1=(K,0)$ is locally asymptotically stable.
\end{proof}

\begin{proof}[Proof of Theorem~\ref{Theorem_R4}]
The nonnegative quadrant $\mathbb{R}_+^2$ is positively invariant. The second equation gives
\begin{equation*}
P(t)=P(0)e^{-mt},
\end{equation*}
so $P(t)\to 0$ as $t\to\infty$. For $N(0)>0$, the first equation's solution satisfies
\begin{equation*}
N(t)=\frac{K N(0)}{N(0)+(K-N(0))e^{-rt}},
\end{equation*}
and therefore $\lim_{t\to\infty}N(t)=K$. Hence every solution with $N(0)>0$ and $P(0)\ge 0$ satisfies $\lim_{t\to\infty}(N(t),P(t))=(K,0)$. Based on Jacobian matrices at $E_0$ and $E_1$, we have $E_0$ is unstable and $E_1$ is locally asymptotically stable. Together with the global convergence above, $E_1=(K,0)$ is globally asymptotically stable for $N(0)>0$ and $P(0)\ge 0$.
\end{proof}

\section{Sliding mode dynamics}
\label{S:sliding_dynamics}
In a piecewise-smooth system, a sliding mode is the motion that occurs when a trajectory reaches a switching boundary and then moves along the boundary instead of crossing it~\cite{bernardo2008piecewise}. The vector fields on either side determine whether trajectories cross the boundary, move away from it, or stay on it and slide.

\subsection{Sliding dynamics on \texorpdfstring{$\Sigma_{(1,2)}$}{Sigma12}}
The manifold $\Sigma_{(1,2)}$ is a discontinuity surface separating the two vector fields $f_1$ and $f_2$. We can compute
\begin{equation*}
\left.\langle f_1,n_1\rangle\right|_{(N,P)\in\Sigma_{(1,2)}}
=
rN_a\left(1-\frac{N_a}{K}\right)-aN_aP,
\qquad
\left.\langle f_2,n_1\rangle\right|_{(N,P)\in\Sigma_{(1,2)}}
=
-rN_a-aN_aP.
\end{equation*}
Since $r>0$, $a>0$, $N_a>0$, and $P>P_a>0$, we have
$\left.\langle f_2,n_1\rangle\right|_{\Sigma_{(1,2)}}<0$. Hence, there is no attracting sliding segment on $\Sigma_{(1,2)}$. On the other hand, $\Sigma_{(1,2)}$ contains an escaping sliding segment
\begin{equation*}
\Sigma_{e,(1,2)}=
\left\{(N,P)\in\mathbb{R}_+^2:
N=N_a,\ 
P_a<P<\frac{r}{a}\left(1-\frac{N_a}{K}\right)
\right\}.
\end{equation*}
Naturally, $\Sigma_{e,(1,2)}$ is nonempty if
$P_a<\frac{r}{a}\left(1-\frac{N_a}{K}\right)$. In this case, the vectors for the two different structures of system~\eqref{rewritten_system} are directed away from each other.

According to Filippov's convex method~\cite{filippov2013differential}, the escaping sliding mode dynamics along $\Sigma_{e,(1,2)}$ of system~\eqref{rewritten_system} is
\begin{equation}
\begin{aligned}
\begin{pmatrix}
\dfrac{dN}{dt} \\[6pt]
\dfrac{dP}{dt}
\end{pmatrix}
&= f_{e,(1,2)}(N,P)=\lambda f_1+(1-\lambda)f_2,
\qquad (N,P)\in\Sigma_{e,(1,2)},\\
&=
\frac{K(r+aP)}{r(2K-N_a)}
\begin{pmatrix}
rN_a\left(1-\dfrac{N_a}{K}\right)-aN_aP\\
baN_aP-mP
\end{pmatrix}
+
\frac{K\left[r\left(1-\dfrac{N_a}{K}\right)-aP\right]}{r(2K-N_a)}
\begin{pmatrix}
-rN_a-aN_aP\\
baN_aP-mP
\end{pmatrix}\\
&=
\begin{pmatrix}
0\\
baN_aP-mP
\end{pmatrix}.
\end{aligned}
\label{sliding1}
\end{equation}
where
\begin{equation*}
    \lambda=\frac{\langle f_2,n_1\rangle}{\langle f_2,n_1\rangle-\langle f_1,n_1\rangle}=\frac{K(r+aP)}{r(2K-N_a)}.
\end{equation*}
The value of $\lambda$ is chosen so that the resulting vector field is tangent to the switching surface. Hence, there is no positive isolated pseudo-equilibrium on $\Sigma_{e,(1,2)}$ unless $N_a=\frac{m}{ba}$, in which case every point of $\Sigma_{e,(1,2)}$ is an equilibrium of the escaping sliding dynamics.

\subsection{Sliding dynamics on \texorpdfstring{$\Sigma_{(2,3)}$}{Sigma23}}
The manifold $\Sigma_{(2,3)}$ is a discontinuity surface separating the two vector fields $f_2$ and $f_3$. We can compute
\begin{equation*}
    \left.\langle f_2,n_2\rangle\right|_{(N,P)\in\Sigma_{(2,3)}}= baNP_a-mP_a, \qquad \left.\langle f_3,n_2\rangle\right|_{(N,P)\in\Sigma_{(2,3)}} = -mP_a.
\end{equation*}
Since $m>0$ and $P_a>0$, we have $\left.\langle f_3,n_2\rangle\right|_{\Sigma_{(2,3)}}=-mP_a<0$. There is no attracting sliding segment on $\Sigma_{(2,3)}$. However, $\Sigma_{(2,3)}$ contains an escaping sliding segment
\begin{equation*}
    \Sigma_{e,(2,3)} = \left\{(N,P)\in\mathbb{R}_+^2: \frac{m}{ba}<N<N_a,\ P=P_a \right\}.
\end{equation*}
$\Sigma_{e,(2,3)}$ is nonempty if $N_a>\frac{m}{ba}$. In this case, the vectors for the two different structures of system~\eqref{rewritten_system} are directed away from each other. $\Sigma_{(2,3)}$ contains an escaping sliding segment.

According to Filippov's convex method~\cite{filippov2013differential}, the escaping sliding mode dynamics along $\Sigma_{e,(2,3)}$ of system~\eqref{rewritten_system} is
\begin{equation}
\begin{aligned}
\begin{pmatrix}
\dfrac{dN}{dt} \\[6pt]
\dfrac{dP}{dt}
\end{pmatrix}
&= f_{e,(2,3)}(N,P)=\lambda f_2+(1-\lambda)f_3,
\qquad (N,P)\in\Sigma_{e,(2,3)},\\
&=
\frac{m}{baN}
\begin{pmatrix}
-rN-aNP_a\\
baNP_a-mP_a
\end{pmatrix}
+
\frac{baN-m}{baN}
\begin{pmatrix}
-rN\\
-mP_a
\end{pmatrix}\\
&=
\begin{pmatrix}
-rN-\dfrac{mP_a}{b}\\
0
\end{pmatrix}.
\end{aligned}
\label{sliding2}
\end{equation}
where
\begin{equation*}
    \lambda = \frac{\langle f_3,n_2\rangle} {\langle f_3,n_2\rangle-\langle f_2,n_2\rangle} = \frac{m}{baN}.
\end{equation*}
Since $r>0$, $m>0$, $b>0$, $N>0$, and $P_a>0$, we have $-rN-\frac{mP_a}{b}<0$. System~\eqref{sliding2} has no positive pseudo-equilibrium on $\Sigma_{e,(2,3)}$.

\subsection{Sliding dynamics on \texorpdfstring{$\Sigma_{(3,4)}$}{Sigma34}}
The manifold $\Sigma_{(3,4)}$ is a discontinuity surface between the two vector fields $f_3$ and $f_4$. We compute
\begin{equation*}
    \left.\langle f_4,n_1\rangle\right|_{(N,P)\in\Sigma_{(3,4)}} = rN_a\left(1-\frac{N_a}{K}\right), \qquad \left.\langle f_3,n_1\rangle\right|_{(N,P)\in\Sigma_{(3,4)}} = -rN_a<0.
\end{equation*}
Hence, there is no attracting sliding segment on $\Sigma_{(3,4)}$. On the other hand, if $N_a<K$, $\Sigma_{(3,4)}$ contains an escaping sliding segment given by
\begin{equation*}
    \Sigma_{e,(3,4)} = \left\{(N,P)\in\mathbb{R}_+^2: N=N_a,\ 0<P<P_a,\ N_a<K \right\}.
\end{equation*}
According to Filippov's convex method~\cite{filippov2013differential}, the escaping sliding mode dynamics along $\Sigma_{e,(3,4)}$ of system~\eqref{rewritten_system} is
\begin{equation}
\begin{aligned}
\begin{pmatrix}
\dfrac{dN}{dt} \\[6pt]
\dfrac{dP}{dt}
\end{pmatrix}
&= f_{e,(3,4)}(N,P)
= \lambda f_4 + (1-\lambda)f_3,
\qquad (N,P)\in \Sigma_{e,(3,4)},\\
&= \frac{K}{2K-N_a}
\begin{pmatrix}
rN_a\left(1-\dfrac{N_a}{K}\right) \\
-mP
\end{pmatrix}
+
\frac{K-N_a}{2K-N_a}
\begin{pmatrix}
-rN_a \\
-mP
\end{pmatrix}\\
&= \begin{pmatrix}
\dfrac{K rN_a\left(1-\frac{N_a}{K}\right)-rN_a(K-N_a)}
{2K-N_a} \\[1em]
\dfrac{-KmP-(K-N_a)mP}{2K-N_a}
\end{pmatrix}\\
&= \begin{pmatrix}
0 \\
-mP
\end{pmatrix}.
\end{aligned}
\label{sliding3}
\end{equation}
where
\begin{equation*}
    \lambda = \frac{\langle f_3,n_1\rangle} {\langle f_3,n_1\rangle-\langle f_4,n_1\rangle} = \frac{K}{2K-N_a}.
\end{equation*}
There is no positive pseudo-equilibrium on $\Sigma_{e,(3,4)}$, since $-mP=0$ implies $P=0$, whereas $0<P<P_a$ on $\Sigma_{e,(3,4)}$.

\subsection{Sliding dynamics on \texorpdfstring{$\Sigma_{(4,1)}$}{Sigma41}}
The manifold $\Sigma_{(4,1)}$ is a discontinuity surface separating the two vector fields $f_1$ and $f_4$. We can compute
\begin{equation*}
    \left.\langle f_1,n_2\rangle\right|_{(N,P)\in\Sigma_{(4,1)}} = baNP_a-mP_a, \qquad \left.\langle f_4,n_2\rangle\right|_{(N,P)\in\Sigma_{(4,1)}} = -mP_a<0.
\end{equation*}
There is again no attracting sliding segment on $\Sigma_{(4,1)}$. However, $\Sigma_{(4,1)}$ contains an escaping sliding segment
\begin{equation*}
    \Sigma_{e,(4,1)} = \left\{(N,P)\in\mathbb{R}_+^2: N>\max\left\{N_a,\frac{m}{ba}\right\},\ P=P_a \right\}.
\end{equation*}
According to Filippov's convex method~\cite{filippov2013differential}, the escaping sliding mode dynamics along $\Sigma_{e,(4,1)}$ of system~\eqref{rewritten_system} is
\begin{equation}
\begin{aligned}
\begin{pmatrix}
\dfrac{dN}{dt} \\[6pt]
\dfrac{dP}{dt}
\end{pmatrix}
&= f_{e,(4,1)}(N,P)=\lambda f_1+(1-\lambda)f_4,
\qquad (N,P)\in\Sigma_{e,(4,1)},\\
&=
\frac{m}{baN}
\begin{pmatrix}
rN\left(1-\dfrac{N}{K}\right)-aNP_a\\
baNP_a-mP_a
\end{pmatrix}
+
\frac{baN-m}{baN}
\begin{pmatrix}
rN\left(1-\dfrac{N}{K}\right)\\
-mP_a
\end{pmatrix}\\
&=
\begin{pmatrix}
rN\left(1-\dfrac{N}{K}\right)-\dfrac{mP_a}{b}\\
0
\end{pmatrix}.
\end{aligned}
\label{sliding4}
\end{equation}
where
\begin{equation*}
    \lambda = \frac{\langle f_4,n_2\rangle} {\langle f_4,n_2\rangle-\langle f_1,n_2\rangle} = \frac{m}{baN}.
\end{equation*}
Then system~\eqref{sliding4} has pseudo-equilibria satisfying
\begin{equation*}
    rN\left(1-\frac{N}{K}\right)-\frac{mP_a}{b}=0.
\end{equation*}
Thus, the candidate pseudo-equilibria are
\begin{equation*}
    E_{p4}^{\pm} = \left(\frac{K\pm\sqrt{K^2-\frac{4KmP_a}{br}}}{2}, P_a \right), \qquad \text{if }  K^2-\frac{4KmP_a}{br}\geq 0.
\end{equation*}
These are pseudo-equilibria on $\Sigma_{e,(4,1)}$ provided that their $N$-coordinates satisfy
\begin{equation*}
    \frac{K\pm\sqrt{K^2-\frac{4KmP_a}{br}}}{2} > \max\left\{N_a,\frac{m}{ba}\right\}.
\end{equation*}

\subsection{Phase portrait showing the dynamics when flows reach a switching boundary}
Figure~\ref{fig:FigureS1} illustrates this behaviour for two representative threshold configurations. The coloured portions of the switching manifolds indicate escaping sliding segments, while the markers denote admissible pseudo-equilibria when they exist. In panel A, $\Sigma_{e,(1,2)}$, $\Sigma_{e,(3,4)}$, and $\Sigma_{e,(4,1)}$ are present. the degenerate condition $N_a=\frac{m}{ba}$ produces a continuum of pseudo-equilibria on $\Sigma_{e,(1,2)}$, together with an isolated pseudo-equilibrium on $\Sigma_{e,(4,1)}$. In panel B, $\Sigma_{e,(2,3)}$ is empty because $N_a<\frac{m}{ba}$. The “gate” on $\Sigma_{e,(4,1)}$ allows trajectories to enter $R_4$ from $R_1$ and subsequently approach $E_1$. For both cases, nearby trajectories are repelled from the escaping segments and subsequently follow the smooth dynamics of the region they enter. Their long-term behaviour is therefore determined by the attracting state of the region, such as extinction, prey-only persistence, or coexistence. Figure~\ref{fig:FigureS2} extends this comparison across a broader range of $N_a$ and $P_a$, showing how the escaping segments and pseudo-equilibria vary across threshold configurations.

\begin{figure}[ht]
    \centering
    \includegraphics[width=0.8\linewidth]{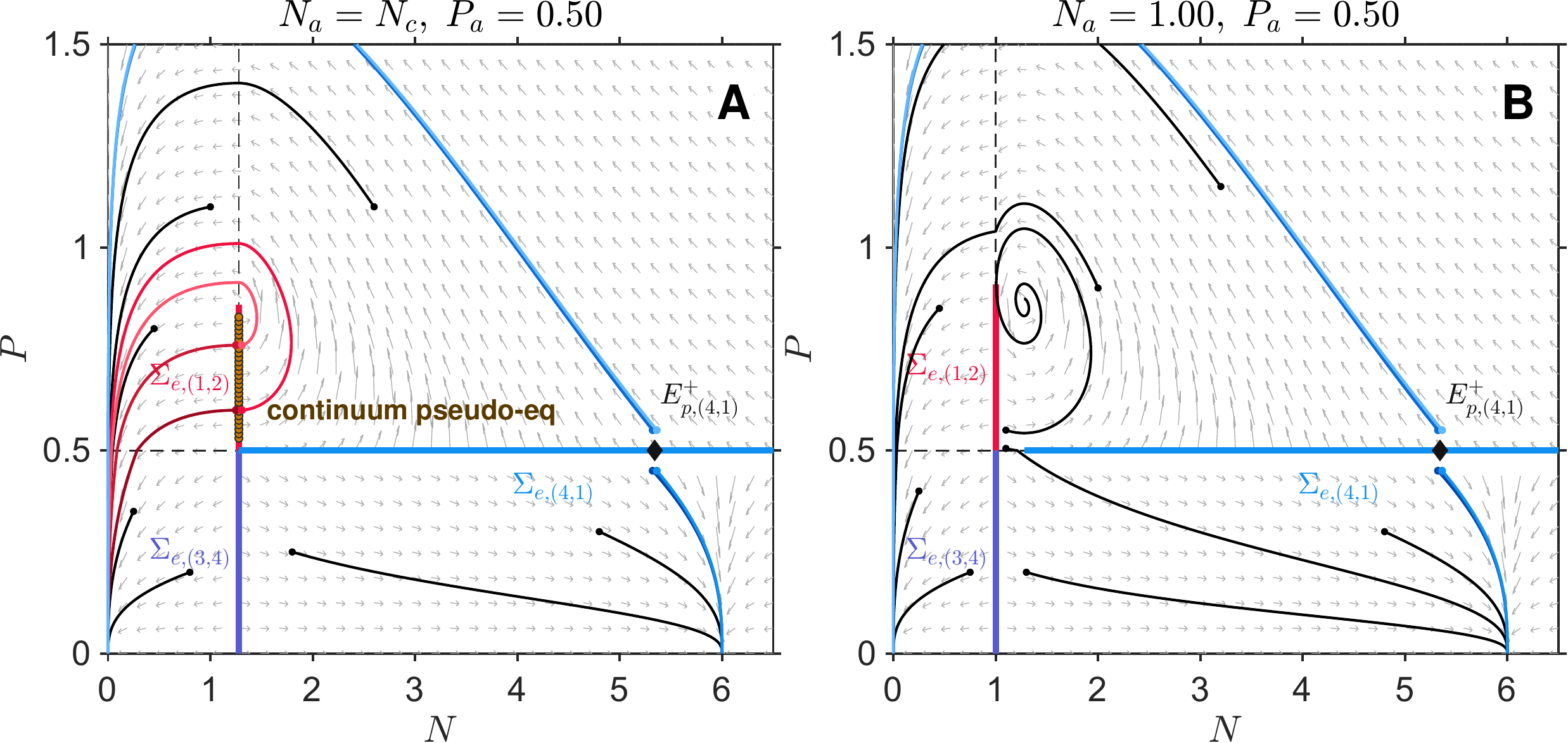}
    \caption{Phase portraits of system~\ref{predator_prey_model} under two Allee-threshold configurations: (A) $N_a=N_c=m/(ba)\approx1.278$, $P_a=0.50$; and (B) $N_a=1.00$, $P_a=0.50$. Dashed vertical and horizontal lines denote the switching thresholds $N=N_a$ and $P=P_a$, respectively, while coloured segments indicate the escaping sliding regions $\Sigma_{e,(1,2)}$, $\Sigma_{e,(3,4)}$, and $\Sigma_{e,(4,1)}$ when they are admissible. The filled diamond denotes the admissible pseudo-equilibrium $E_{p,(4,1)}^{+}$. In panel A, where $N_a=N_c$, the points along $\Sigma_{e,(1,2)}$ form a continuum of pseudo-equilibria. Black curves show representative trajectories, and coloured trajectories illustrate the local dynamics near the continuum pseudo-equilibria and $E_{p,(4,1)}^{+}$. Gray arrows indicate the background vector field. Other parameters are $r=1.2$, $K=6.0$, $a=1.1$, $b=0.32$, and $m=0.45$.}
    \label{fig:FigureS1}
\end{figure}

\subsection{Example of the sliding mode under different \texorpdfstring{$(N_a,P_a)$}{NaPa}}
\begin{figure}[ht]
    \centering
    \includegraphics[width=\textwidth]{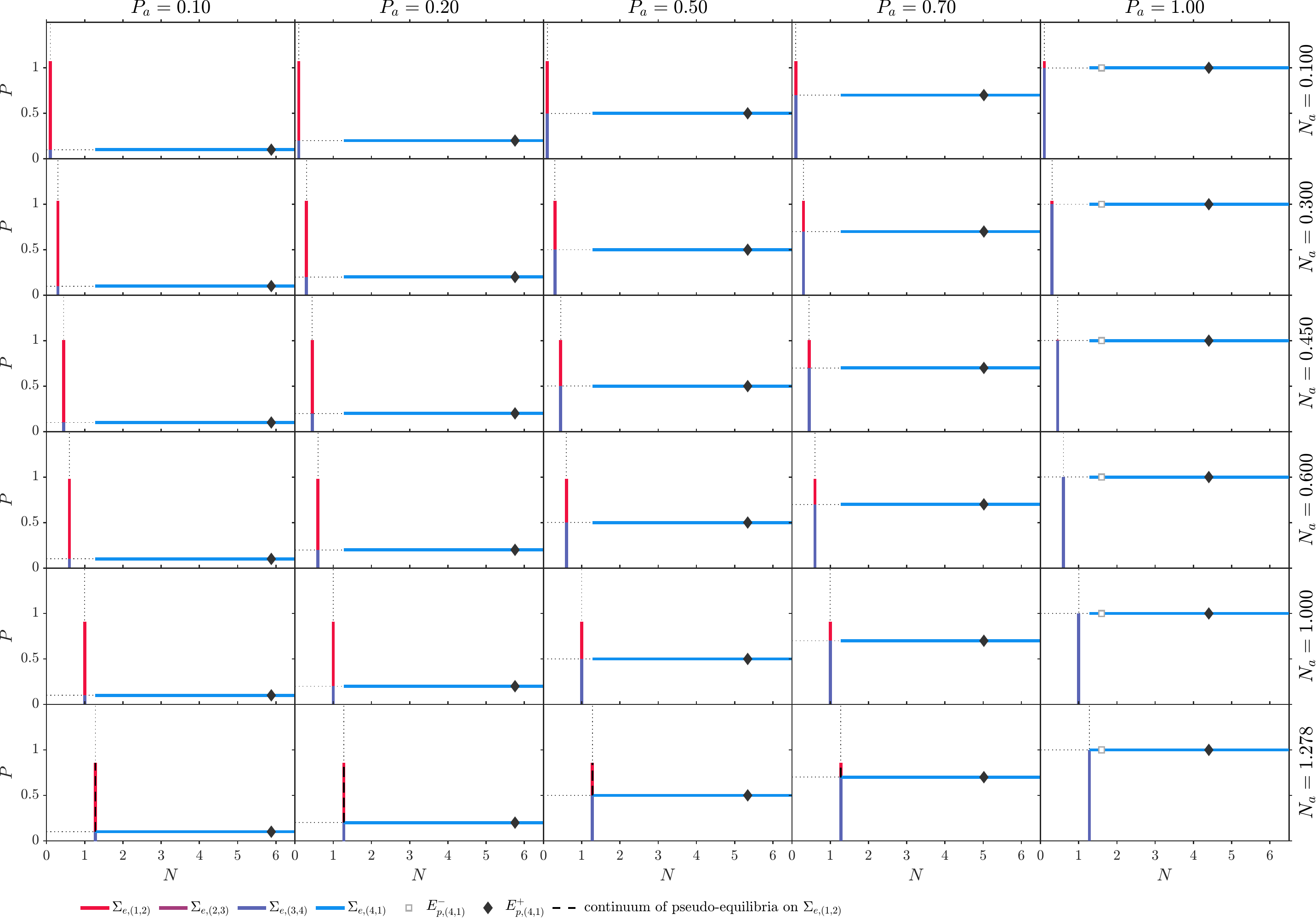}
    \caption{
    Escaping sliding segments and pseudo-equilibria of system~\eqref{rewritten_system}
    for different values of $N_a$ and $P_a$. The colored curves represent the escaping sliding segments $\Sigma_{e,(1,2)}$, $\Sigma_{e,(2,3)}$, $\Sigma_{e,(3,4)}$, and $\Sigma_{e,(4,1)}$. The markers denote the pseudo-equilibria $E_{p,(4,1)}^{-}$ and $E_{p,(4,1)}^{+}$ when they exist, while the dashed curve represents the continuum of pseudo-equilibria on $\Sigma_{e,(1,2)}$ when $N_a=m/(ba)$.}
    \label{fig:FigureS2}
\end{figure}

\newpage
\bibliography{sample}

\end{document}